%% file: main.tex
\documentclass{IEEEtran}
\usepackage{cite}
\usepackage{amsmath,amssymb,amsfonts}
\usepackage{amsthm}
\usepackage{graphicx}
\usepackage{textcomp}
\usepackage{tikz}
\usetikzlibrary{positioning,arrows.meta,calc}
\usepackage{mathtools}
\usepackage{algorithm}
\usepackage{algpseudocode}
\usepackage{booktabs}
\usepackage{xcolor}

\def\BibTeX{{\rm B\kern-.05em{\sc i\kern-.025em b}\kern-.08em
    T\kern-.1667em\lower.7ex\hbox{E}\kern-.125emX}}
\begin{document}
\title{Tube MPC for Bilinear Koopman Models using Robust Control Contraction Metrics}
\author{Thomas O. de Jong, \IEEEmembership{Graduate Student, IEEE},
        Mircea Lazar, \IEEEmembership{Senior Member, IEEE}
\thanks{Thomas O. de Jong is with the Department of Electrical Engineering, Eindhoven University of Technology, 5600 MB Eindhoven, The Netherlands (e-mail: t.o.d.jong@tue.nl).}
\thanks{Mircea Lazar is with the Department of Electrical Engineering, Eindhoven University of Technology, 5600 MB Eindhoven, The Netherlands (e-mail: m.lazar@tue.nl).}}

\maketitle

\newtheorem{theorem}{Theorem}
\newtheorem{problem}{Problem}
\newtheorem{lemma}{Lemma}
\newtheorem{proposition}{Proposition}
\newtheorem{corollary}{Corollary}
\newtheorem{definition}{Definition}
\newtheorem{remark}{Remark}
\newtheorem{assumption}{Assumption}
\newtheorem{example}{Example}

\newcommand{\col}{\operatorname{col}}
\newcommand{\diag}{\operatorname{diag}}

\begin{abstract}
This paper presents a robust tube model predictive control (MPC) framework for nonlinear systems represented by bilinear Koopman models identified from data. We derive discrete-time robust control contraction metrics (RCCMs) for bilinear Koopman models, which certify a contraction property that explicitly accounts for the mismatch between the Koopman model and the true dynamics. To synthesize such certificates for the typically high-dimensional lifted state, we propose a scalable learning approach that trains neural network parameterizations of the metric and of an associated feedback controller over a sparsely sampled subset of the lifted state space. The resulting certificates are used to construct a robust tube MPC scheme with tightened constraints, for which we establish recursive feasibility, robust constraint satisfaction and input-to-state stability (ISS) of the closed-loop system with respect to the mismatch between the Koopman model and the true dynamics. By exploiting the linear parameter-varying structure of bilinear Koopman models, the MPC problem becomes convex for a fixed scheduling sequence, while the geodesic computation underlying the contractive feedback is reformulated, via a Chebyshev pseudospectral discretization, as a sequence of quadratic programs. The complete framework is validated on a nonlinear pendulum benchmark with a state-dependent input gain.
\end{abstract}


\section{Introduction}
\label{sec:introduction}
Model predictive control (MPC) is one of the few control methodologies that systematically handles hard state and input constraints, with a mature stability theory in both nominal and robust settings \cite{MAYNE2000789,mayne2005robust}. Its implementation, however, hinges on the availability of a sufficiently accurate prediction model. In many applications, first-principles modeling is costly or impractical, which motivates learning prediction models directly from data. Koopman operator theory \cite{koopman1931hamiltonian} offers an attractive route: the state is lifted through a set of observable functions into a higher-dimensional space in which the dynamics evolve (approximately) linearly, so that tools from linear systems theory become applicable to nonlinear systems. For control-affine nonlinear systems, bilinear Koopman realizations are particularly appealing, since an exact linear lifted realization with inputs generally does not exist, whereas bilinear realizations can approximate control-affine dynamics with arbitrary accuracy \cite{bruder2021advantages}, with certified approximation error bounds recently derived in \cite{strasser2024safedmd}.

In practice, finite data and a finite-dimensional lifting map render modeling errors unavoidable, and a Koopman model used for prediction and control should therefore be regarded as an uncertain model, which calls for robust MPC formulations. For linear Koopman realizations, robust designs are available, including tube-based MPC \cite{zhang2022robust} and Koopman data-driven predictive control with robust stability and recursive feasibility guarantees \cite{de2024koopman}. For bilinear realizations, in contrast, robust MPC schemes that jointly provide recursive feasibility, robust constraint satisfaction and closed-loop guarantees for the true nonlinear dynamics remain scarce, in part because the high dimension of the lifted state renders classical robust invariant set constructions intractable. In parallel, a recent line of work has established closed-loop guarantees for Koopman-based MPC via certified error bounds on the learned surrogate: practical asymptotic stability under uniform bounds \cite{worthmann2024terminal,bold2025datadriven}, strengthened to asymptotic stability of the origin under \emph{proportional} error bounds that vanish at the equilibrium \cite{schimperna2025asymptotic,schimperna2026terminal}, with certified proportional and kernel-based bounds available for bilinear surrogates \cite{strasser2024safedmd,strasser2025kernel}; see \cite{strasser2026overview} for a comprehensive overview.

Contraction theory \cite{LOHMILLER1998683} studies stability of nonlinear systems in terms of the convergence of neighboring trajectories, and control contraction metrics (CCMs) \cite{manchester2017control} turn this differential viewpoint into a convex synthesis framework for stabilizing feedback design. Robust CCMs \cite{manchester2018robust} additionally provide an explicit disturbance gain, which makes them natural certificates for tube-based robust MPC: the Riemannian distance induced by the metric bounds the deviation between the true closed-loop trajectory and a nominal MPC trajectory, see, e.g., \cite{ZhaoGridding,sasfi2023robust,YangCTubeDT2024}. Existing contraction-based tube MPC schemes, however, predominantly address continuous-time dynamics and rely on sum-of-squares (SOS) programming or gridding-based LMI relaxations for metric synthesis \cite{sasfi2023robust,ZhaoGridding}, while discrete-time formulations \cite{wei2021control,YangCTubeDT2024} have been developed for low-dimensional states. Neither ingredient carries over directly to Koopman models: the lifted state is high-dimensional by construction, which rules out SOS synthesis, and the discrete-time setting couples the metric at successive time instants, which complicates both synthesis and online geodesic computation.

Motivated by the above, this paper develops a discrete-time RCCM-based tube MPC framework tailored to data-driven bilinear Koopman models, formulated directly in the lifted state space. The bilinear structure is exploited twice. First, it renders the robust contraction condition a linear matrix inequality at every fixed lifted state and input, which enables learning neural network parameterizations of the metric and of an associated feedback gain over a sparsely sampled subset of the lifted space, in the spirit of neural contraction metrics \cite{li2025neural}. Second, it admits an exact linear parameter-varying (LPV) embedding of the nominal prediction model, which renders the tube MPC problem convex for a fixed scheduling sequence, while the geodesic computation underlying the contractive feedback is discretized via Chebyshev pseudospectral methods \cite{gong2009chebyshev,leung2017nonlinear} and reformulated as a sequence of quadratic programs. Notably, the resulting closed-loop guarantees take the form of input-to-state stability (ISS) \cite{JIANG2001857} with respect to the Koopman model mismatch, which is stronger than the input-to-state practical stability (ISpS) and set-convergence guarantees typically established for contraction-based and Koopman-based tube MPC \cite{mayne2005robust,sasfi2023robust,YangCTubeDT2024,zhang2022robust}. Our main contributions are as follows:

\begin{itemize}
\item[\textbf{(C1)}] We develop a robust control contraction metric (RCCM)-based tube MPC framework for discrete-time bilinear Koopman models with recursive feasibility, stability and constraint satisfaction of the true dynamics.

\item[\textbf{(C2)}] We combine discrete-time RCCMs with pseudospectral geodesic approximations to reformulate the tube MPC as a sequence of convex problems and geodesic computation as sequential quadratic programs (SQPs).

\item[\textbf{(C3)}] We propose a sparse sampling strategy for learning neural contraction metrics and controllers for bilinear Koopman models.

\item[\textbf{(C4)}] We design terminal ingredients using LPV-based techniques and model Koopman prediction errors as bounded disturbances. Under this assumption, we establish recursive feasibility and closed-loop ISS guarantees for the original nonlinear system with respect to the Koopman model mismatch.
\end{itemize}

\emph{Concurrent work:} While finalizing this manuscript, the concurrent work \cite{higuchi2026bilinear} appeared, which also combines data-driven bilinear Koopman models with contraction-based tube MPC. Therein, predicted lifted states are reprojected onto the original state space at every prediction step, and a nominal contraction metric with an additive error bound is synthesized in the original, low-dimensional state space, e.g., via SOS programming for polynomial observables, while the tube anchoring constraint embeds a Riemannian distance into the resulting nonlinear MPC problem. In contrast, our approach operates directly in the lifted space and employs a robust contraction inequality with an explicit disturbance gain. This makes the complexity of the learned contraction certificates independent of the specific complexity of the Koopman lifting functions, although it scales with the dimension of the lifted state. In comparison, the alternative approach avoids a direct dependence on the number of basis functions, but its complexity instead depends on the complexity of the individual lifting functions. We design terminal sets via LPV techniques rather than a terminal point constraint, and we exploit the LPV structure of the bilinear model to split the online computations into a convex MPC problem and a sequence of quadratic programs for the geodesic. The two approaches are therefore complementary in terms of certificate structure, synthesis scalability and online computational architecture. Moreover, we establish ISS of the closed-loop system with respect to the model mismatch, which is stronger than the convergence to a neighborhood of the target certified in \cite{higuchi2026bilinear}.

The remainder of the paper is organized as follows. Section~\ref{sec:preliminaries} reviews discrete-time robust control contraction theory. Section~\ref{sec:koopman} presents the considered bilinear Koopman model class and the associated learning problem. Section~\ref{sec:rccm_koopman} derives the robust contraction inequality for bilinear Koopman models and the proposed sparse-sampling neural certificate learning approach. Section~\ref{sec:mpc} presents the robust MPC scheme together with recursive feasibility, constraint satisfaction and closed-loop input-to-state stability guarantees. Section~\ref{sec:computation} develops the computational framework and Section~\ref{sec:example} provides a numerical validation, while conclusions are summarized in Section~\ref{sec:conclusion}.

\subsection*{Notations and basic definitions}
Let $\mathbb{R}$ and $\mathbb{N}$ denote the field of real and natural numbers. For every $c\in\mathbb{R}$ and $\Pi\subseteq\mathbb{R}$ define $\Pi_{\geq c}:= \{ k\in\Pi | k\geq c\}$. 
Throughout this paper, for any finite number $q\in\mathbb{N}_{\geq1}$ of column vectors or  functions$\{\xi_1,\dots,\xi_q\}$ we will make use of the operator $\text{col}(\xi_1,\dots,\xi_q):=\begin{bmatrix}
\xi_1^\top & \dots & \xi_q^\top    
\end{bmatrix}^\top$. For any vector $v \in\mathbb{R}^q$, we use $v_i$ to denote the $i$-th element of $v$. Here, $A_{i,j}$ denotes the element of $A$ in the $i$-th row and $j$-th column. For any finite number $q\in\mathbb{N}_{\geq 1}$ of scalars or scalar–valued functions $\{d_1,\dots,d_q\}$ we define the operator $D := \diag(d_1,\dots,d_q)\in\mathbb{R}^{q\times q}$ as the diagonal matrix satisfying $D_{i,i} = d_i$ for all $i=1,\dots,q$ and $D_{i,j}=0$ whenever $i\neq j$.  The vectorization operator $\operatorname{vec}(\cdot)$, for a matrix $A \in \mathbb{R}^{m \times n}$ is defined as \[
\operatorname{vec}(A)
:=
\begin{bmatrix}
A_{1,1} & \cdots & A_{m,1} &
\cdots &
A_{1,n} & \cdots & A_{m,n}
\end{bmatrix}^{\top}.
\]
Let $v=\text{vec}(A)$. A feedforward neural network is defined as a composition of affine maps and pointwise nonlinearities. Given an input $x \in \mathbb{R}^{n_0}$, the hidden states are recursively defined by
\begin{align}
h_0 &= x, \nonumber \\
h_\ell &= \sigma(W_\ell h_{\ell-1} + b_\ell),
\quad \ell = 1,\dots,L, \nonumber \\
f_\theta(x) &= W_{L+1} h_L + b_{L+1}.
\label{eq:nn_blueprint}
\end{align}

Here, $n_\ell \in \mathbb{N}$ for $\ell=0,\dots,L+1$, with
$W_\ell \in \mathbb{R}^{n_\ell\times n_{\ell-1}}$ and
$b_\ell \in \mathbb{R}^{n_\ell}$.
The class of feedforward neural networks is defined as
\[
\mathcal{NN}
:=
\left\{
f_\theta : \mathbb{R}^{n_0}\to\mathbb{R}^{n_{L+1}}
\,\middle|\,
\theta=\{(W_\ell,b_\ell)\}_{\ell=1}^{L+1}
\right\},
\]
where $f_\theta$ is given by \eqref{eq:nn_blueprint}. Throughout this work, the superscript $\cdot^*$ denotes optimized network parameters obtained through training. Accordingly, a trained neural network is denoted by $f_{\theta^*}$, where $\theta^*={(W_\ell^*,b_\ell^*)}_{\ell=1}^{L+1}$.

A continuous function \(\alpha:\mathbb{R}_{\geq 0}\rightarrow\mathbb{R}_{\geq 0}\) is a class-\(\mathcal{K}\) function if \(\alpha(s)>0\) for all \(s>0\), it is strictly increasing, and \(\alpha(0)=0\). It is furthermore a class-\(\mathcal{K}_{\infty}\) function if
\[
\lim_{s\rightarrow\infty}\alpha(s)=\infty.
\]
A continuous function
\[
\beta:\mathbb{R}_{\geq 0}\times\mathbb{Z}_{\geq 0}\rightarrow\mathbb{R}_{\geq 0},
\]
is a class-\(\mathcal{KL}\) function if \(\beta(s,t)\) is a class-\(\mathcal{K}\) function with respect to \(s\) for all fixed \(t\), it is strictly decreasing in \(t\) for all \(s>0\), and
\[
\lim_{t\rightarrow\infty}\beta(s,t)=0,\quad \forall s>0.
\]

\section{Preliminaries on Robust Control Contraction Theory}
\label{sec:preliminaries}
This section reviews discrete-time robust control contraction theory for a general class of uncertain nonlinear models, which provides the certificates used throughout the paper to bound the deviation between true and nominal closed-loop trajectories.
We consider nonlinear MIMO systems with discrete-time dynamics
\begin{align}\label{eqn:system}
x_{k+1} = f(x_k,u_k),
\end{align}
where \(u_k\in\mathbb{U}\subseteq\mathbb{R}^{n_u}\) is the input and
\(x_k\in\mathbb{X}\subseteq\mathbb{R}^{n_x}\) denotes the system state and $k\in\mathbb{N}$ represents a time instant.
Assume an approximate model is available:
\begin{align}\label{eqn:system_model}
x_{k+1} = \hat{f}(x_k,u_k) + w_k,
\end{align}
where \(w_k\in\mathbb{W}\subseteq\mathbb{R}^{n_x}\) represents the modeling error,
given by
\begin{align}\label{eqn:model_error}
w_k = \hat{f}(x_k,u_k)-f(x_k,u_k).
\end{align}
We assume that this modeling error is unknown and is part of a compact set $\mathbb{W}$ with zero in its interior $0\in\mathbb{W}$. In the following, we denote by
\begin{align}\label{eqn:max_model_error}
\bar{w}:= \sup\{\|w\|_2 : w \in\mathbb{W}\},
\end{align}
and by
\begin{align}\label{eqn:max_mismatch}
\bar{\delta}:= \sup\{\|w_1-w_2\|_2 : w_1,w_2 \in\mathbb{W}\}.
\end{align}
\begin{figure*}[t!]
    \centering
    \begin{minipage}{0.48\linewidth}
        \centering
        \includegraphics[width=\linewidth]{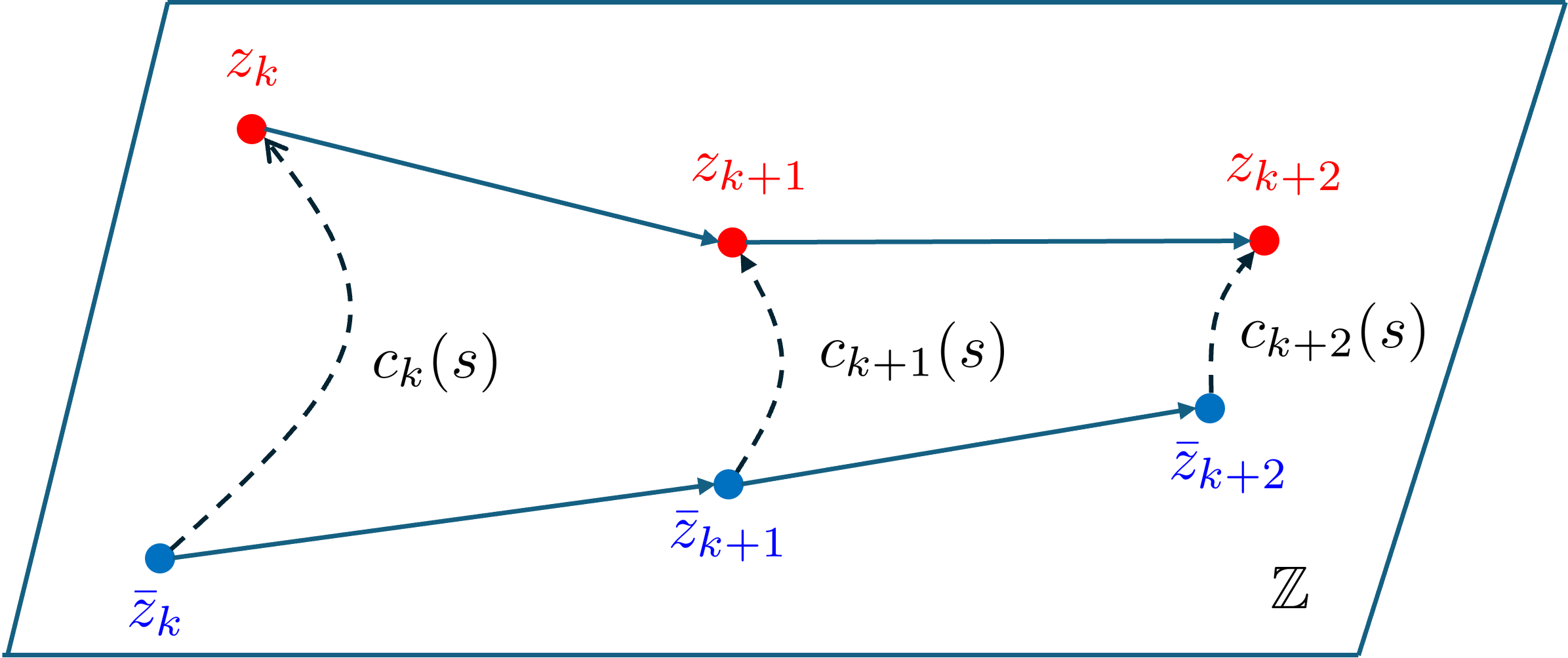}
    \end{minipage}
    \hfill
    \begin{minipage}{0.48\linewidth}
        \centering
        \includegraphics[width=\linewidth]{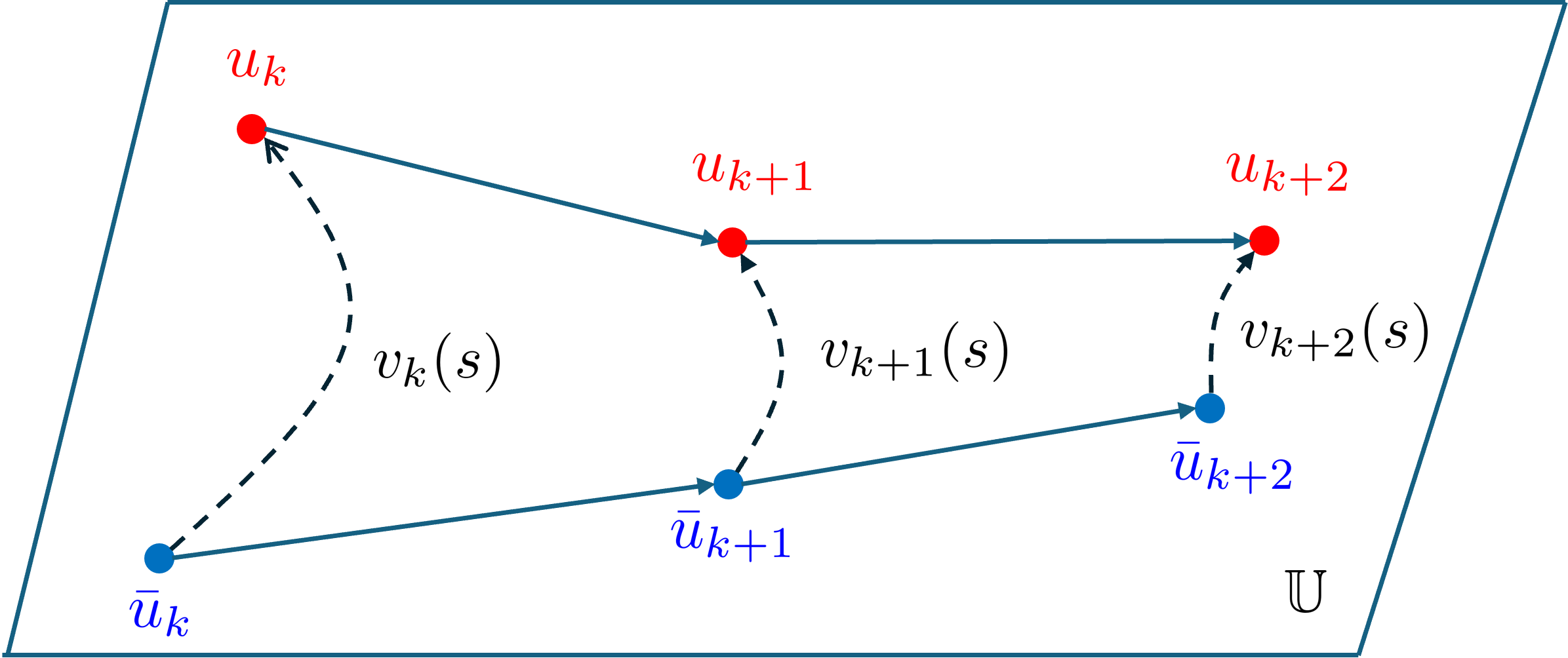}
    \end{minipage}
    \caption{Smooth paths connecting admissible state trajectories $z_k,\bar{z}_k$ and inputs $u_k,\bar{u}_k$.}
    \label{fig:d_dynamics}
\end{figure*}
Next, we introduce several relevant concepts from control contraction metric (CCM) theory \cite{LOHMILLER1998683}, with particular emphasis on robust formulations \cite{manchester2018robust} and the discrete-time setting; see, e.g., \cite{wei2021control}. In the CCM theory, we study stability of a nonlinear dynamical system \eqref{eqn:system_model} with respect to trajectories. In particular, we are interested in studying the stability properties between two arbitrary trajectories satisfying $x_{k+1} := \hat{f}(x_k,u_k) + w_k$, i.e., 
\begin{align*}
    &\mathbf{x} = (x_0,x_1,\dots), \quad \mathbf{u} = (u_0,u_1,\dots) , \quad \mathbf{w} = (w_0,w_1,\dots), \\
    &\bar{\mathbf{x}} = (\bar{x}_0,\bar{x}_1,\dots), \quad \bar{\mathbf{u}} = (\bar{u}_0,\bar{u}_1,\dots) , \quad \bar{\mathbf{w}} = (\bar{w}_0,\bar{w}_1,\dots).
\end{align*}
Next, consider an arbitrary smooth path $c_k:[0,1]\rightarrow \mathbb{R}^{n_x}$ and  $v_k:[0,1]\rightarrow \mathbb{R}^{n_u}$ which connects these two trajectories for the states and inputs. These smooth paths are defined such that $c_k(0) = \bar{x}_k$, $c_k(1) = x_k$ and $v_k(0) = \bar{u}_k$, $v_k(1) = u_k$ for all $k\in\mathbb{N}$. We can propagate forward this smooth path in time using the model  \eqref{eqn:system_model}, as follows
\begin{subequations}\label{eqn:smooth_path}
\begin{align}
& c_{k+1}(s) = \hat{f}(c_{k}(s),v_k(s)) + s w_k + (1-s)\bar{w}_k, \label{eqn:propagated_smooth_path_a}\\
& v_k(s) = \bar{u}_k + \int_0^s K(c_{k}(s))\frac{\partial c_{k}(s)}{\partial s} ds.
\end{align}
\end{subequations}
Note that by multiplying the disturbance related terms by $s$ and $(1-s)$ we get exactly the trajectory related to the case when $\bar{w}_k$ when $s=0$ and $w_k$ when $s=1$. A visualization is given in Figure~\ref{fig:d_dynamics}. By taking the derivative of \eqref{eqn:propagated_smooth_path_a} with respect to $s$ we obtain
\begin{align}\label{eqn:differential_dynamics}
    \frac{\partial c_{k+1}(s)}{\partial s} = \left( A_k + B_k K(c_k(s))\right) \frac{\partial c_k(s)}{\partial s} + w_k - \bar{w}_k,
\end{align}
where 
\begin{subequations}\label{eqn:koopman_differential}
    \begin{align}
    &A_k := \frac{\partial \hat{f}(c_k(s),v_k(s))}{\partial c_k(s)}, B_k := \frac{\partial \hat{f}(c_k(s),v_k(s))}{\partial v_k(s)}.
\end{align}
\end{subequations}
A contraction metric is defined as a positive definite, locally bounded matrix-valued function that satisfies the following conditions:
\begin{align}\label{eqn:boundedness metric}
    \alpha_1 I \preceq M(x) \preceq \alpha_2 I, \quad x\in\mathbb{X},
\end{align}
where $M:\mathbb{R}^{n_x}\rightarrow\mathbb{R}^{n_x\times n_x}$ and $\alpha_2 \geq \alpha_1 > 0$ are the smallest and largest eigenvalues for all $x\in\mathbb{X}$. Next, along the smooth curve $c_k(s)$, we define the Riemannian distance $d_{c_k}(\bar{x}_k,x_k)$ and Riemannian energy $E_{c_k}(\bar{x}_k,x_k)$ as
\begin{subequations}
    \begin{align}
        &d_{c_k}(\bar{x}_k,x_k) := \int_0^1 \sqrt{\frac{\partial c_k(s)}{\partial s}^\top M(c_k(s)) \frac{\partial c_k(s)}{\partial s} } ds, \\
        &E_{c_k}(\bar{x}_k,x_k) := \int_0^1 \frac{\partial c_k(s)}{\partial s}^\top  M(c_k(s)) \frac{\partial c_k(s)}{\partial s}  ds. \label{eqn:riemannian_energy}
    \end{align}
\end{subequations}
Furthermore, the minimum length curve or geodesic $\gamma_k$, connecting any two points is defined as 
\begin{align}\label{eqn:geodesic}
    \gamma_k(s) := \arg \min_{c_k} d_{c_k}(\bar{x}_k,x_k).
\end{align}
Using the geodesic, we define the feedback control law
\begin{subequations}\label{eqn:controller}
\begin{align}
\kappa(\bar x_k,x_k)
&:=
\int_0^1
K(\gamma_k(s))
\frac{\partial\gamma_k(s)}{\partial s}\,ds,
\\
u_k
&=
\bar u_k+\kappa(\bar x_k,x_k).
\end{align}    
\end{subequations}
We say that the system \eqref{eqn:system_model} possesses the robust contraction property when 
\begin{align}\label{eqn:robust_property}
    E_{\gamma_{k+1}}(\bar{x}_{k+1},x_{k+1}) \leq (1-\alpha) E_{\gamma_k}(\bar{x
    }_k,x_k) + \nu \|w_k- \bar{w}_k\|_2^2.
\end{align}
Property \eqref{eqn:robust_property} bounds the Euclidean distance between any pair of trajectories of the system \eqref{eqn:system_model} as presented next.
\begin{lemma}[Euclidean error bound]\label{lemma:1}
For the system \eqref{eqn:system_model}, let inequality \eqref{eqn:robust_property} hold with the bounded metric \eqref{eqn:boundedness metric}. Then, the Euclidean distance between any two trajectories satisfies
\begin{align}
\label{eqn:2norm_bound}
\|x_{k+N}-\bar x_{k+N}\|_2
&\le
\sqrt{\frac{\alpha_2}{\alpha_1}}
(1-\alpha)^{N/2}
\|x_k-\bar x_k\|_2 \nonumber
\\
&\quad+
\sqrt{\frac{\nu}{\alpha_1}}
\sum_{i=0}^{N-1}
(1-\alpha)^{i/2}
\bar{\delta}.
\end{align}
\end{lemma}
\begin{proof}
  For notational convenience, we suppress the function arguments and define $E_{k+N}:=E_{\gamma_{k+N}}(\bar{x}_{k+N},x_{k+N})$ and $E_k:= E_{\gamma_k}(\bar{x}_k,x_k)$. Applying the recursive inequality \eqref{eqn:robust_property} repeatedly over \(N\) time steps yields
\begin{align*}
&E_{k+N} \\
&\leq (1-\alpha)E_{k+N-1}
    +\nu\|w_{k+N-1}\|_2^2 \\
&\leq (1-\alpha)^2E_{k+N-2}
    +(1-\alpha)\nu\|w_{k+N-2}-\bar{w}_{k+N-2}\|_2^2 \\
    &\qquad \qquad \qquad \qquad \qquad \qquad +\nu\|w_{k+N-1} - \bar{w}_{k+N-1}\|_2^2 \\
&\hspace{2cm}\vdots \\
&\leq (1-\alpha)^N E_k
    +\sum_{i=0}^{N-1}
    \nu(1-\alpha)^{i}\|w_{k+N-1-i}-\bar{w}_{k+N-1-i}\|_2^2.
\end{align*}
Next, we derive a lower bound on the Riemannian energy in terms of the Euclidean distance. First, since $\alpha_1 I \preceq M(x) \preceq \alpha_2 I$ for all $x\in\mathbb{X}$, it holds that 
\[
\int_0^1
\alpha_1
\left\|
\frac{\partial \gamma_{k+N}(s)}{\partial s}
\right\|_2^2
\,ds \leq E_{k+N}.
\]
Since \(\alpha_1 I\) is a constant metric, the straight line
\[
c_{k+N}(s)=(1-s)\bar x_{k+N}+sx_{k+N}
\]
minimizes the corresponding energy. Therefore,
\[
E_{k+N}
\geq
\int_0^1
\alpha_1
\left\|
\frac{\partial c_{k+N}(s)}{\partial s}
\right\|_2^2
\,ds
=
\alpha_1
\|x_{k+N}-\bar x_{k+N}\|_2^2.
\]
Similarly, an upper bound on $E_k$ is obtained by first evaluating the energy along the straight line
\[
c_k(s)=(1-s)\bar x_{k}+sx_{k},
\]
which upper bounds the geodesic energy, and subsequently using the metric bound \(M(x)\preceq \alpha_2 I\). This yields
\begin{align*}
E_{k}
&\leq
\int_0^1
\frac{\partial c_k(s)}{\partial s}^{\!\top}
M
\frac{\partial c(s)}{\partial s}
\,ds \leq
\int_0^1
\alpha_2
\left\|
\frac{\partial c(s)}{\partial s}
\right\|_2^2
\,ds \\
&=
\alpha_2
\|x_{k}-\bar x_{k}\|_2^2.
\end{align*}
Combining the previous results yields
\begin{align*}
\|x_{k+N}-\bar x_{k+N}\|_2^2
&\leq
\frac{\alpha_2}{\alpha_1}(1-\alpha)^N
\|x_k-\bar x_k\|_2^2 \\
& \hspace{-10mm} +\frac{\nu}{\alpha_1}
\sum_{i=0}^{N-1}
(1-\alpha)^{i}
\|w_{k+N-1-i}-\bar{w}_{k+N-1-i}\|_2^2.
\end{align*}
Taking square roots and using the subadditivity of the square root gives
\begin{align*}
\|x_{k+N}-\bar x_{k+N}\|_2
&\leq
\sqrt{\frac{\alpha_2}{\alpha_1}}\,
(1-\alpha)^{N/2}
\|x_k-\bar x_k\|_2 \\
& \hspace{-15mm} +
\sqrt{\frac{\nu}{\alpha_1}}
\sum_{i=0}^{N-1}
(1-\alpha)^{\frac{i}{2}}
\|w_{k+N-1-i}-\bar{w}_{k+N-1-i}\|_2.
\end{align*}
Finally, we upper bound the disturbance related terms by $\bar{\delta}$, which completes the proof.
\end{proof}
\begin{remark}
Taking the limit $N\rightarrow\infty$ in \eqref{eqn:2norm_bound}, the first term vanishes since \(1-\alpha<1\). The remaining term becomes a geometric series, which gives
\[
\limsup_{i\rightarrow\infty}\|x_i-\bar{x}_i\|_2
\leq
\sqrt{\frac{\nu}{\alpha_1}}
\frac{\bar{\delta}}{1-\sqrt{1-\alpha}}.
\]
Thus, the distance between the two trajectories converges to a bounded value determined by the disturbance magnitude and the contraction properties.
\end{remark}

Next, we present the relevant conditions for computing contraction metrics $M$ and controllers $K$. For notational convenience, we suppress the function arguments and define
\begin{align*}
&M^{+} := M\left(\hat{f}(x_k,u_k)+ w_k\right), \\
&W^{+} := W\left((\hat{f}(x_k,u_k)+ w_k\right),
\end{align*}
where $W = M^{-1}$ is the inverse of the metric.
\begin{assumption}\label{assumption:1}
There exist a smooth matrix-valued function $W:\mathbb{X} \rightarrow\mathbb{R}^{n_x\times n_x}$, continuous function $L:\mathbb{Z}\rightarrow \mathbb{R}^{n_u\times n_x}$, a contraction rate 
\(\alpha\in(0,1)\), a disturbance gain $\nu>0$  and positive scalars \(\alpha_1,\alpha_2\) satisfying 
\(0<\alpha_1\leq \alpha_2\), such that
\[
\frac{1}{\alpha_2} I \preceq W(x_k) \preceq \frac{1}{\alpha_1}  I,
\]
for all $x_k\in\mathbb{X}$, $u_k\in\mathbb{U}$ and $w_k\in\mathbb{W}$. Furthermore, the following matrix inequality holds:
\begin{align}\label{eqn:contraction_inequality}
\begin{bmatrix}
    W^+ & AW + BL & I \\
    \left(AW + BL\right)^\top & (1-\alpha)W & 0 \\
    I & 0 & \nu I
\end{bmatrix}
\succeq 0 .
\end{align}
\end{assumption}
\noindent Next we present the robust control contraction metrics result.

\begin{theorem}\label{thrm:robust_contraction}
Consider model \eqref{eqn:system_model} with the differential dynamics \eqref{eqn:koopman_differential}, under the feedback controller \eqref{eqn:controller} with $K = L M$ and contraction metric $M = W^{-1}$. If Assumption~\ref{assumption:1} holds, then the system satisfies the robust contraction property, i.e., inequality \eqref{eqn:robust_property} holds.
\end{theorem}
\begin{proof}
The proof follows that of \cite{YangCTubeDT2024} and is adapted to the bilinear Koopman setting. 
Premultiplying and postmultiplying \eqref{eqn:contraction_inequality} by $\operatorname{diag}(I,M,I)$ yields
$$
\begin{bmatrix}
    W^+ & A + BK & I \\
    (A+BK)^\top & (1-\alpha)M & 0 \\
    I & 0 & \nu I
\end{bmatrix}
\succeq 0,
$$
where we used the identities \(M=W^{-1}\) and \(K=LM\). Applying the Schur complement and subsequently multiplying both sides of the resulting inequality by $-1$, gives
$$
\begin{bmatrix}
   A_{cl}^\top M^+A_{cl} - (1-\alpha)M &
   A_{cl}^\top M^+\\
    M^+A_{cl} &
   M^+ - \nu I
\end{bmatrix}
\preceq 0,
$$
where $A_{cl}:=(A+BK)$ and the matrix inequality reverses direction because multiplication by the negative scalar \(-1\). Next, multiply from the left and right with $\col\left(\frac{\partial c_k(s)}{\partial s}, w_k-\bar{w}_k\right)$, and substitute $c_k(s)$ into $M(c_k(s))$ where $c_k(s)$ is a smooth path between two arbitrary points $x_k$, $\bar{x}_k \in \mathbb{X}$ and $c_k(s)\in\mathbb{X}$ for all $s\in[0,1]$. This gives
\begin{align*}
& \frac{\partial c_{k+1}(s)}{\partial s}^{\top} M(c_{k+1}(s))
\frac{\partial c_{k+1}(s)}{\partial s} \\
& \quad-
(1-\alpha)
\frac{\partial c_{k}(s)}{\partial s}^{\top} M(c_k(s))
\frac{\partial c_{k}(s)}{\partial s}
-
\nu\|w_k-\bar{w}_k\|_2^2
\leq 0,   
\end{align*}
where \(c_{k+1}(s)\) is the smooth path obtained by propagating \(c_k(s)\) through the model according to \eqref{eqn:differential_dynamics}. Consequently, \(c_{k+1}(s)\) connects \(x_{k+1}\) and \(\bar{x}_{k+1}\).
Integrating the above inequality with respect to \(s\) over the interval \([0,1]\) and using the definition of the Riemannian energy in \eqref{eqn:riemannian_energy} yields
\[
E_{c_{k+1}}(\bar{z}_{k+1},z_{k+1})
\leq
(1-\alpha)E_{c_k}(\bar{z}_k,z_k)
+
\nu\|w_k-\bar{w}_k\|_2^2.
\]
Without loss of generality, we choose \(c_k(s)=\gamma_k(s)\), where \(\gamma_k(s)\) is the geodesic connecting \(z_k\) and \(\bar{z}_k\). Note that the propagated path \(c_{k+1}(s)\) is not necessarily the geodesic connecting \(x_{k+1}\) and \(\bar{x}_{k+1}\). Nevertheless, since the geodesic minimizes the Riemannian distance by definition \eqref{eqn:geodesic}, it also minimizes the Riemannian energy and we have
\[
E_{\gamma_{k+1}}(\bar{z}_{k+1},z_{k+1})
\leq
E_{c_{k+1}}(\bar{z}_{k+1},z_{k+1}).
\]
Therefore, inequality \eqref{eqn:robust_property} holds which finalizes the proof.
\end{proof} 

\begin{remark}[Relation to incremental stability]\label{rem:delta_iss}
The robust contraction property \eqref{eqn:robust_property}, together with Lemma~\ref{lemma:1}, constitutes an incremental ISS ($\delta$-ISS) certificate \cite{angeli2002lyapunov}: it bounds the deviation between \emph{any} two trajectories by a $\mathcal{KL}$-term in their initial distance plus a gain acting on the disturbance difference. Incremental stability of learned models is also the enabling property in recent NMPC designs for recurrent neural network models \cite{bonassi2024nonlinear,schimperna2024robust}, where $\delta$-ISS of the model is verified a posteriori via algebraic conditions. In contrast, in our work, the incremental certificate and the associated tracking feedback are synthesized jointly via the RCCM.
\end{remark}

The next section will deal with the problem of learning an approximate model of the form \eqref{eqn:system_model} for general nonlinear systems \eqref{eqn:system} using the Koopman framework \cite{koopman1931hamiltonian}.

\section{Bilinear Koopman model learning}\label{sec:koopman}
This section develops a bilinear Koopman model from data by introducing the lifted representation and formulating a joint learning problem for the lifting function and model matrices. The resulting model provides the foundation for the robust control design developed in the following sections.

\subsection{Bilinear Koopman models}
To obtain a tractable representation for prediction and control, we employ the Koopman framework \cite{koopman1931hamiltonian}, in which the nonlinear state is lifted into a higher-dimensional space of observables. In particular, we define the lifted state
\begin{align}\label{eqn:Koopman_lifting}
    z_k := \bar{\phi}(x_k)
:= \col(x_k,\phi_{n_x+1}(x_k),\dots,\phi_{n_z}(x_k)),
\end{align}
where $z_k\in\mathbb{Z}\subseteq \mathbb{R}^{n_z}$ and \(n_z>n_x\). In this work, we include the original state in the lifting vector. Although this may not yield an exact Koopman-invariant representation, it is suitable since the Koopman model is used as an approximate model of the nonlinear system. The resulting finite-dimensional bilinear Koopman model is given by
\begin{align}\label{eqn:Koopman_approximation_true}
z_{k+1}
&=
Az_k + B_0 u_k + \sum_{i=1}^{n_u}[u_k]_i B_i z_k + w_k \\
&=: f_K(z_k, u_k,w_k), \nonumber
\end{align}
where \(w_k\in\mathbb{W}\) represents the mismatch between the nonlinear lifted dynamics and the finite-dimensional bilinear Koopman approximation, i.e., 
\begin{align}\label{eqn:approximation_error}
w_{k} := \bar{\phi}(x_{k+1}) -  A\bar{\phi}(x_k) - B_0 u_k - \sum_{i=1}^{n_u}[u_k]_i B_i \bar{\phi}(x_k).
\end{align}
We assume that this approximation error is bounded and belongs to a compact disturbance set \(\mathbb{W}\), satisfying \(w_k \in \mathbb{W}\) and \(0 \in \mathbb{W}\). In this work, $\mathbb{W}$ is characterized as a uniform norm ball estimated from data; alternative certified characterizations of the approximation error of bilinear Koopman surrogates include proportional bounds that vanish at the origin \cite{strasser2024safedmd} and deterministic kernel-based bounds \cite{strasser2025kernel}; see also \cite{strasser2026overview}.

\subsection{Bilinear Koopman Model Learning}
To train a bilinear Koopman model of the form \eqref{eqn:Koopman_approximation_true} we collected a state-input data set:
\begin{align}\label{eqn:dataset_model}
    \mathcal{D} = \{x_k,u_k\}_{k=0}^{N_s+N-2}.
\end{align}
Next, we parameterize the lifting function \eqref{eqn:Koopman_lifting} as 
\begin{align}\label{eqn:nn_lifting}
z_k := \bar{\phi}_\theta(x_k) = \col(x_k,f_\theta(x_k)),
\end{align} 
where $f_\theta\in\mathcal{NN}$ is a feedforward neural network with parameter set $\theta$. The bilinear Koopman model learning Problem is presented next. 
\begin{problem}[Bilinear Koopman model learning]\label{prob:model_learning}
    \begin{subequations}
        \begin{align}
            \min_{\theta,A,B_0,\dots,B_{n_u}}& \frac{1}{N N_s}  \sum_{i=1}^{N_s}\sum_{k=1}^{N} \frac{\|\hat{z}^{(i)}_k  - z^{(i)}_k\|_2^2}{\|z^{(i)}_k\|_2^2} + \lambda \mathcal{L}_{reg}\\
            &\hspace{-2em}\text{subject to:} \nonumber \\
            &\hspace{-2em}\hat{z}^{(i)}_{k+1} = A \hat{z}^{(i)}_{k} + B_0 u_{i+k} + \sum_{j=1}^{n_u}[ u_{k+i}]_j B_j \hat{z}^{(i)}_{k}, \\
           &\hspace{-2em}\hat{z}^{(i)}_{0} = \begin{bmatrix}
                x_i \\ f_{\theta}(x_i)
            \end{bmatrix},  \quad \hat{z}^{(i)}_{k} = \begin{bmatrix}
                x_{i+k} \\ f_{\theta}(x_{i+k})
            \end{bmatrix}.
        \end{align}
    \end{subequations}
\end{problem}
\noindent The regularization component of the objective function is given by
\[
\mathcal{L}_{reg}
=  \|\theta\|_2^2
+ \|A\|_F^2
+ \|B_0\|_F^2 + \dots 
+ \|B_{n_u}\|_F^2.
\]
where the parameter norm is defined as
\[
\|\theta\|_2^2 := \sum_{\ell=1}^{L+1} \left( \|W_\ell\|_F^2 + \|b_\ell\|_2^2 \right).
\]
Problem~\ref{prob:model_learning} simultaneously trains the lifting function $\bar{\phi}_{\theta}:=\col(x_k,f_{\theta}(x_k))$ and the model matrices $A$, $B_0,\dots,B_{n_u}$ while minimizing the multi-step prediction error over a prediction horizon $N\in\mathbb{N}$. The trained Koopman model will be denoted as:
\begin{align}\label{eqn:Koopman_approximation_trained}
z_{k+1}
&=
A^*z_k+ B^*_0 u_k + \sum_{i=1}^{n_u}[u_k]_i B^*_i z_k + w_k \\
&=: f^*_K(z_k, u_k,w_k), \nonumber
\end{align}
where $z_0 := \bar{\phi}_{\theta^*}(x_0)$.

Given the trained bilinear Koopman model \eqref{eqn:Koopman_approximation_trained}, two ingredients remain to be constructed before a robust MPC scheme can be formulated: a robust contraction certificate for the lifted dynamics, developed in Section~\ref{sec:rccm_koopman}, and a corresponding tube-based constraint tightening with terminal ingredients, developed in Section~\ref{sec:mpc}.

\section{Robust Control Contraction Theory Applied to Bilinear Koopman Models}
\label{sec:rccm_koopman}
For the trained bilinear Koopman model \eqref{eqn:Koopman_approximation_trained}, we are interested in studying the stability properties between 
two arbitrary trajectories satisfying $z_{k+1} = f^*_K(z_k,u_k,w_k)$, i.e., 
\begin{align*}
    &\mathbf{z} = (z_0,z_1,\dots), \quad \mathbf{u} = (u_0,u_1,\dots) , \quad \mathbf{w} = (w_0,w_1,\dots), \\
    &\bar{\mathbf{z}} = (\bar{z}_0,\bar{z}_1,\dots), \quad \bar{\mathbf{u}} = (\bar{u}_0,\bar{u}_1,\dots) , \quad \bar{\mathbf{w}} = (\bar{w}_0,\bar{w}_1,\dots).
\end{align*}
Next, consider an arbitrary smooth path $c_k:[0,1]\rightarrow \mathbb{R}^{n_z}$ and  $v_k:[0,1]\rightarrow \mathbb{R}^{n_u}$ which connects these two trajectories for the states and inputs. These smooth paths are defined such that $c_k(0) = \bar{z}_k$, $c_k(1) = z_k$ and $v_k(0) = \bar{u}_k$, $v_k(1) = u_k$ for all $k\in\mathbb{N}$. We can propagate forward this smooth path in time using the model  \eqref{eqn:system_model}, as follows
\begin{subequations}\label{eqn:smooth_path_koopman}
\begin{align}\label{eqn:propagated_smooth_path_koopman}
& c_{k+1}(s) = A^*c_{k}(s) + B^*_0 v_{k}(s) + \sum_{i=1}^{n_u}[v_{k}(s)]_i B^*_i c_{k}(s) \nonumber  \\
& \qquad \qquad \qquad\qquad \qquad\qquad +  s w_k + (1-s)\bar{w}_k, \\\
& v_k(s) = \bar{u}_k + \int_0^s K(c_{k}(s))\frac{\partial c_{k}(s)}{\partial s} ds.
\end{align}
\end{subequations}
Note that by multiplying the disturbance related terms by $s\in[0,1]$ we get exactly the nominal trajectory related to the case when $\bar{w}_k=0$ when $s=0$ and $w_k$ when $s=1$. A visualization is given in Figure~\ref{fig:d_dynamics}. By taking the time derivative with respect to $s$ we obtain
\begin{align}\label{eqn:differential_dynamics_koopman}
    \frac{\partial c_{k+1}(s)}{\partial s} = \left( A^*(v_k(s)) + B^*(c_k(s)) K(c_k(s)) \right) \frac{\partial c_k(s)}{\partial s} \nonumber \\+ w_k - \bar{w}_k
\end{align}
where 
\begin{subequations}\label{eqn:koopman_differential_z}
    \begin{align}
    &A^*(v_k(s)) := \left( A^* + \sum_{i=1}^{n_u} B^*_i [v_k(s)]_i\right), \\
    &B^*(c_k(s)) := \left( B^*_0 + \begin{bmatrix}
    B^*_1c_k(s) & \dots & B^*_{n_u}c_k(s) 
\end{bmatrix} \right).
\end{align}
\end{subequations}
A contraction metric is defined as a positive definite, uniformly bounded matrix-valued function that satisfies the following conditions:
\begin{align}\label{eqn:boundedness_metric_z}
    \alpha_1 I \preceq M(z) \preceq \alpha_2 I, \quad z\in\mathbb{Z}
\end{align}
where $M:\mathbb{R}^{n_z}\rightarrow\mathbb{R}^{n_z\times n_z}$ and $\alpha_2 \geq \alpha_1 > 0$ are the smallest and largest eigenvalues. The resulting contraction inequality becomes:
\begin{align}
\label{eqn:contraction_inequality_koopman}
\begin{bmatrix}
W(z_{k+1})
& A^*(u_k)W(z_k)+B^*(z_k)L(z_k)
& I\\
*
& (1-\alpha)W(z_k)
& 0\\
*
& *
& \nu I
\end{bmatrix}
\succeq 0,
\end{align}
which must hold for all $z_k\in\mathbb{Z}$, $u_k\in\mathbb{U}$ and $w_k\in\mathbb{W}$ where the disturbance enters through $z_{k+1} = f_K(z_k,u_k,w_k)$. 
\begin{remark}\label{rem:contraction_meetric_solve}
For any fixed state $z_k$ and $u_k$, inequality \eqref{eqn:contraction_inequality_koopman} is a convex linear matrix inequality (LMI) in the values of the metric $W(z)$ and controller $L(z)$. The synthesis problem is nevertheless infinite-dimensional, as the inequality must hold for all $z_k\in\mathbb{Z}$,  $u_k\in\mathbb{U}$ and  $w_k\in\mathbb{W}$. A common approach is to parameterize $W$ and $L$ by polynomials and enforce the condition using sum-of-squares (SOS) programming; see, e.g., \cite{wei2021control}. This is possible in this specific case due to the bilinearity of the Koopman model. However, SOS-based methods typically scale poorly with increasing state dimension and polynomial degree, making it not ideal for Koopman approaches which typically have a large state. Alternatively, one may discretize the state and input spaces and solve the resulting finite collection of LMIs, see, e.g., \cite{ZhaoGridding}. To improve scalability while retaining expressive function approximations, a promising alternative is to parameterize the metric \(W\) and controller \(L\) using neural networks and enforce the contraction condition over a representative set of sampled states during training; see, e.g., \cite{li2025neural}.
\end{remark}

\subsection{Neural Contraction Metric Learning}
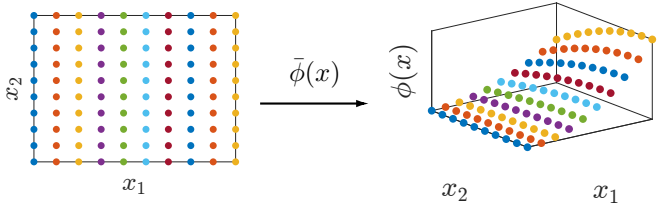
\begin{figure}[t]
\centering
\input{figures/metric_learning_sparse}
\caption{Illustration of the lifting map $\bar{\phi}$ from the state space to the lifted state space. The image of the lifting map typically occupies only a sparse subset of the lifted space $\mathbb{R}^{n_z}$.}
\label{fig:sparse_sampling}
\end{figure}
As discussed in Remark~\ref{rem:contraction_meetric_solve}, a key challenge in computing a contraction metric and controller is that the contraction inequality \eqref{eqn:contraction_inequality_koopman} must hold for all \(z\in\mathbb{Z}\) and \(u\in\mathbb{U}\). This becomes particularly challenging for Koopman models, which typically have a high-dimensional lifted state space. The lifted state set is defined as the image of the state constraint set under the trained lifting function:
\begin{align}
    \mathbb{Z}
    :=
    \left\{
    z\in\mathbb{R}^{n_z}
    :
    z=\bar{\phi}_{\theta^*}(x),\;
    x\in\mathbb{X}
    \right\}.
\end{align}
Although \(\mathbb{Z}\subseteq\mathbb{R}^{n_z}\), it typically occupies only a small subset of the lifted space, as illustrated in Figure~\ref{fig:sparse_sampling}. This observation motivates the proposed sparse sampling strategy, which is combined with a neural network parameterization of the contraction metric and controller to efficiently learn contraction certificates for bilinear Koopman models.

To learn a contraction metric for the trained Koopman model $f^*_K$ we generate an artificial data set
\begin{align}\label{eqn:dataset_metric}
    \mathcal{D}_M = \{\hat{x}_k,\hat{u}_k,\hat{w}_k\}_{k=0}^{N_M},
\end{align}
where \(\hat{x}_k\), \(\hat{u}_k\), and \(\hat{w}_k\) are generated such that the sets \(\mathbb{X}\), \(\mathbb{U}\), and \(\mathbb{W}\) are covered as extensively as possible. The \(\hat{\cdot}\) notation indicates that these samples are virtual points and therefore can be selected arbitrarily without requiring simulations of the actual system. The corresponding lifted states are obtained as $\hat{z}_k:=\bar{\phi}_{\theta^*}(\hat{x}_k)$. To account for the fact that the robust MPC controller will operate in a neighborhood around the lifted manifold, additional off-manifold samples are generated by perturbing the lifted states:
\[
\hat{z}_{k,j}=\hat{z}_{k}+\epsilon_{k,j},
\]
where \(\epsilon_{k,j}\) is a small perturbation. These augmented samples allow the contraction metric and controller to be learned not only on the lifted state manifold, but also in a neighborhood around it, which is required since we will allow for deviations of the measured initial condition, see the robust MPC Problem~\ref{Prob:robust_MPC}. Next, the Koopman model is used to propagate the samples according to
\[
\hat{z}_{k,j}^{+}=f^*_K(\hat{z}_{k,j},\hat{u}_k,\hat{w}_k).
\]
The metric learning problem is then formulated as follows.
\begin{problem}[Metric learning]\label{prob:metric_learning}
\begin{subequations}
\begin{align}
&\min_{\theta_w,\theta_l}\quad 
\sum_{i=1}^{N_M}\sum_{j=1}^{N_\epsilon}\sum_{q=1}^{n_z}
\big(\mathrm{softplus}(\varepsilon - \lambda_q(M_{i,j}))\big)^2 
+ \gamma_c \mathcal{L}_{cond}
\\
&\text{s.t.}\quad \nonumber \\
& M_{i,j} = \nonumber \\
&\begin{bmatrix}
            W_{\theta_w}(\hat{z}^+_{i,j}) & A^*(\hat{u}_i) W_{\theta_w}(\hat{z}_{i,j}) + B^*(\hat{z}_{i,j}) L_{\theta_l}(\hat{z}_{i,j}) & I \\
            * & (1-\alpha)W_{\theta_w}(\hat{z}_{i,j}) & 0 \\
            *& * & \nu I
        \end{bmatrix} \label{eqn:contraction_conbstraint}
    \end{align}
\label{eqn:prob_2_1}
\end{subequations}
\end{problem}
\noindent Above, we parameterize the metric as
$$
W_{\theta_w}(z) := L_{\theta}(z)L^{\top}_{\theta}(z) + \gamma  I, \quad \gamma> 0,
$$
where $L_{\theta}(z)\in\mathbb{R}^{L\times L}$ is a lower triangular matrix
\begin{align}
     L_{\theta}(z) = \begin{bmatrix}
         [\ell_{\theta_\ell}]_1 & 0  & \dots & 0 \\
         [\ell_{\theta_\ell}]_2 & [\ell_{\theta_\ell}]_3  & \dots & 0 \\
         \vdots & \vdots & \ddots & \vdots \\
         [\ell_{\theta_\ell}]_{\frac{2+L(L-1)}{2}} & [\ell_{\theta_\ell}]_{\frac{4+L(L-1)}{2}} & \dots & [\ell_{\theta_\ell}]_{\frac{L(L+1)}{2}} \\
     \end{bmatrix},
\end{align}
constructed from $\ell_{\theta_\ell} : \mathbb{R}^{n_z} \to \mathbb{R}^{\frac{n_z(n_z+1)}{2}}$, which is a feedforward neural network, i.e. $\ell_{\theta_\ell}\in\mathcal{NN}$ with parameter set $\theta_\ell$ and input $z$. This guarantees that $W_{\theta_w}$ is positive definite and symmetric. The softplus function is defined as $\text{softplus}(\beta-x) := \ln(1 + \exp(\beta-x))$. A visualization of the softplus function is given in Figure~\ref{fig:softplus}. By using this function in the cost we penalize negative eigenvalues. As a result, positive definiteness of $M_{i,j}$ is promoted.
\begin{figure}[t]
\centering
\includegraphics[width=0.6\linewidth]{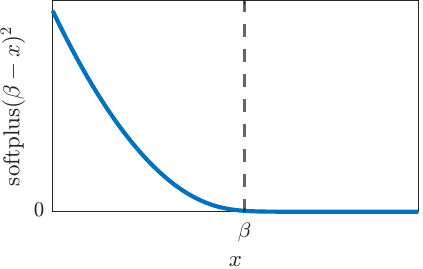}
\caption{Visualization of the softplus cost function used to enforce positive definiteness.}
\label{fig:softplus}
\end{figure}
An additional regularization term $\mathcal{L}_{cond}$ is added to minimize the conditioning number of the metric as this will define the tube size later. The regularizer that is used is given by: 
\begin{align}\label{eqn:conditioning_penalty}
    \mathcal{L}_{cond} = &\big(\beta - \lambda_j( W_{\theta_w}(\hat{z}_i))\big)^2,
\end{align}
This regularizer minimizes the conditioning number, but $\beta>0$ has to be chosen carefully. The minimization of the conditioning number is crucially for obtaining a tight error bound between trajectories, see Lemma~\ref{lemma:1} which is key for the robust MPC scheme that will be presented in the next section. The complete procedure is summarized in Algorithm~\ref{alg:sparse_sampling}.

\begin{algorithm}[t]
\caption{\emph{Offline} neural contraction metric learning}
\label{alg:sparse_sampling}
\begin{algorithmic}[1]
\Require Trained Koopman model $f^*_K$, lifting map $\bar{\phi}_{\theta^*}$, perturbation magnitudes $\delta^{\epsilon}_{j}>0$, $j=1,\ldots,N_{\epsilon}$, where $N_{\epsilon}\in\mathbb{N}$, state, input and disturbance sets $\mathbb{X},\mathbb{U},\mathbb{W}$.
\State Generate a sparse virtual dataset
\[
\mathcal{D}_M=\{\hat{x}_i,\hat{u}_i,\hat{w}_i\}_{i=1}^{N_M},
\]
covering $\mathbb{X}$, $\mathbb{U}$ and $\mathbb{W}$.

\For{each sample $(\hat{x}_i,\hat{u}_i,\hat{w}_i)\in\mathcal{D}_M$}
    \State Compute the lifted state
    \[
    \hat{z}_i=\bar{\phi}_{\theta^*}(\hat{x}_i).
    \]
    \State Generate locally perturbed lifted states
    \[
    \hat{z}_{i,j}=\hat{z}_i+\epsilon_{i,j},
    \qquad
    \epsilon_{i,j}\sim\mathcal{U}([-\delta^{\epsilon}_{j},\delta^{\epsilon}_{j}]^{n_z}),
    \]
    for $j=1,\dots,N_{\epsilon}$.
    \State Propagate the perturbed lifted states using the trained Koopman model
    \[
    \hat{z}_{i,j}^{+}
    =
    f^*_K(\hat{z}_{i,j},\hat{u}_i,\hat{w}_i).
    \]
\EndFor

\State Parameterize the contraction metric $W_{\theta_w}$ and controller $L_{\theta_l}$ using neural networks.

\State Solve Problem~\ref{prob:metric_learning}.

\State \Return trained metric $W_{\theta^*_w}$ and controller $L_{\theta^*_l}$.
\end{algorithmic}
\end{algorithm}

\section{Bilinear Koopman Robust MPC Scheme and Theoretical Guarantees}
\label{sec:mpc}
This section presents the main result of the paper: a robust MPC scheme for bilinear Koopman models with recursive feasibility, constraint satisfaction and closed-loop input-to-state stability guarantees. The complete framework is presented as a diagram in Figure~\ref{fig:control_architecture}.
\subsection{Robust MPC problem}
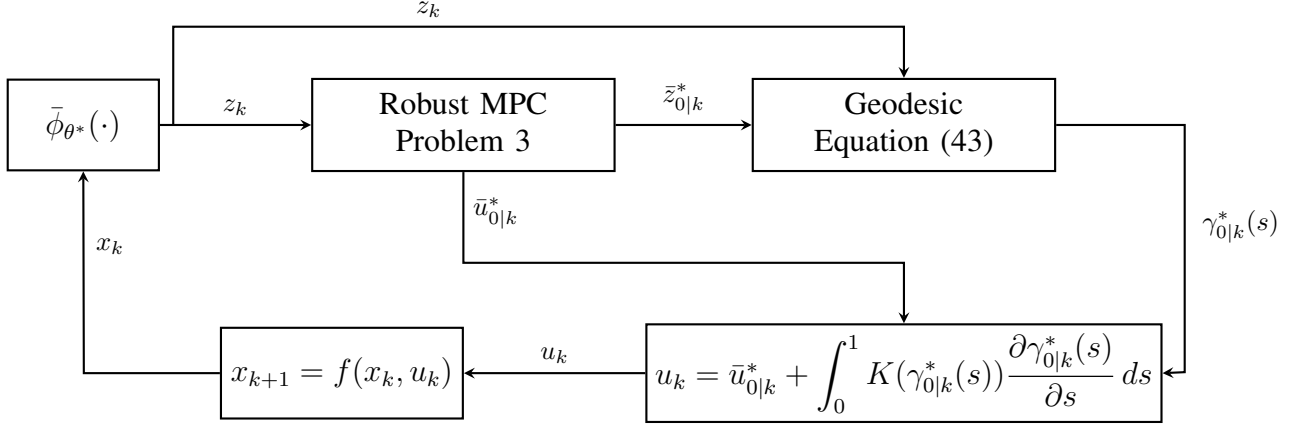
\begin{figure*}[t]
\centering
\input{figures/diagram}
\caption{Architecture of the proposed robust MPC scheme with geodesic-based contractive feedback implementation.}
\label{fig:control_architecture}
\end{figure*}
Consider the trained bilinear Koopman dynamics \eqref{eqn:Koopman_approximation_trained} with zero disturbance, i.e.,  $\bar{z}_{k+1} := f^*_K(\bar{z}_k,\bar{u}_k,0)$. Next, treating the state $\bar{z}_k$ as a scheduling variable, the dynamics can be written as
\begin{align}
    \bar{z}_{k+1}
    =
    A^*\bar{z}_k
    +
    \left(
        B^*_0 +
        \begin{bmatrix}
            B^*_1 \bar{z}_k & \dots & B^*_{n_u} \bar{z}_k
        \end{bmatrix}
    \right)\bar{u}_k .
\end{align}
This can be described in the \emph{exact} LPV form
\begin{align}
    \bar{z}_{k+1}
    =
    A^*\bar{z}_k + \bar{B}^*(\rho_k)\bar{u}_k,
\end{align}
where the scheduling variable $\rho_k =\bar{z}_k$. Let $\boldsymbol{\rho}_k:= \col(\rho_{0|k},\dots,\rho_{N-1|k})$, then 
\begin{align}\label{eqn:koopman_LPV}
\bar{\mathbf{z}}_k = \Phi  z_{0|k} + \Gamma(\boldsymbol{\rho}_k) \bar{\mathbf{u}}_k,
\end{align}
where the multi-step prediction matrices are given by :
\begin{subequations}
    \begin{align}\label{eqn:multi-step_matrices}
    &\Phi = \begin{bmatrix}
       A^* \\
       A^{*^2} \\
        \vdots \\
        A^{*^N}
    \end{bmatrix}, \\
    &\Gamma(\boldsymbol{\rho}_k) = \nonumber \\
    &\begin{bmatrix}
        \bar{B}^*(\rho_{0|k}) & \mathbf{0}  & \dots & \mathbf{0}\\
        A^*\bar{B}^*(\rho_{0|k}) & \bar{B}^*(\rho_{1|k})  & \dots & \mathbf{0}  \\ 
        \vdots & \vdots  & \ddots & \vdots \\
        A^{*^{N-1}} \bar{B}(\rho_{0|k}) & A^{*^{N-2}} \bar{B}^*(\rho_{1|k}) & \dots & \bar{B}^*(\rho_{N-1|k}) 
    \end{bmatrix}.
\end{align} 
\end{subequations}
We now present the MPC problem.
\begin{problem}[Robust MPC for bilinear Koopman models]\label{Prob:robust_MPC}
\begin{subequations}
\begin{align}
\min_{z_{0|k}, \bar{\mathbf{z}}_k,\bar{\mathbf{u}}_k}\quad
&\ell_N(\bar{z}_{N|k},r^z)
+\sum_{i=0}^{N-1}
\ell(\bar{z}_{i|k},r^z,\bar{u}_{i|k},r^u)
\\
\text{s.t.}\quad
&\bar{\mathbf{z}}_k
=
\Phi\bar{z}_{0|k}
+
\Gamma(\boldsymbol{\rho}_k)\bar{\mathbf{u}}_k,
\\ 
& \|z_{0|k} - z^*_{1|k-1}\|_2 \leq \bar{w}, \label{eqn:robust_MPC_c3} \\
&h^u(\bar{u}_{i|k})
+
c^u\delta_{\text{max}}
\leq \mathbf{0},
\\
&h^z(\bar{z}_{i|k})
+
c^z\delta_{\text{max}}
\leq \mathbf{0},
\\
& h_T(z) \leq \mathbf{0}.
\end{align}
\end{subequations}
\end{problem}
The constraint \eqref{eqn:robust_MPC_c3} is required to achieve recursive feasibility and stability, as it ensures that the sequence of optimized initial conditions is a valid trajectory of the uncertain Koopman model \eqref{eqn:Koopman_approximation_true}; its role is made precise in the proof of Theorem~2. The constant $\delta_{\text{max}}$ is given by:
\begin{align}\label{eqn:tightening_constant}
    \delta_{\text{max}} := \frac{\sqrt{\nu} \bar{\delta}}{1-\sqrt{1-\alpha}}.
\end{align}
The functions $h^z:\mathbb{R}^{n_z}\rightarrow \mathbb{R}^{N_z}$, $h^u:\mathbb{R}^{n_u}\rightarrow \mathbb{R}^{N_u}$ and $h_T:\mathbb{R}^{n_z}\rightarrow \mathbb{R}^{N_z}$ are the functions defining the input, state and terminal constraint sets, i.e.,
\begin{subequations}
\begin{align}
\mathbb{Z}
&:= \left\{
z\in\mathbb{R}^{n_z}
\mid
h^z(z)\leq \mathbf{0}
\right\},\\
\mathbb{U}
&:= \left\{
u\in\mathbb{R}^{n_u}
\mid
h^u(u)\leq \mathbf{0}
\right\}, \\
\mathbb{Z}_T
&:= \left\{
z\in\mathbb{R}^{n_z}
\mid
h_T(z)\leq \mathbf{0}
\right\}.
\end{align}    
\end{subequations}
The tightened constraints are defined as:
\begin{subequations}
\begin{align}
\bar{\mathbb{Z}}
&:= \left\{
z\in\mathbb{R}^{n_z}
\mid
h^z(z) + c^z\delta_{\text{max}}\leq \mathbf{0}
\right\},\\
\bar{\mathbb{U}}
&:= \left\{
u\in\mathbb{R}^{n_u}
\mid
h^u(u)+ c^u_j\delta_{\text{max}}\leq \mathbf{0}
\right\}.
\end{align}    
\end{subequations}
The variable $\delta_{\text{max}}$ together with constants $c^z\in\mathbb{R}^{N_z}$ and $c^u\in\mathbb{R}^{N_u}$ determine the constraint tightening. The constants are given by 
\begin{subequations}\label{eqn:constraint_tightening_constant}
    \begin{align}
     &c^z_j := \max_{\bar{z},\in\mathbb{Z}} \left\| \frac{\partial h^z_j}{\partial z}\bigg\vert_{\bar{z}}  M(\bar{z})^{1/2} \right\|, \label{eqn:conditions_constraint_satisfaction_1} \\
             &c^u_j := \max_{(\bar{z},\bar{u})\in\mathbb{Z}\times \mathbb{U}} \left\|  \frac{\partial h^u_j}{\partial u}\bigg\vert_{\bar{u}} K(\bar{z}) M(\bar{z})^{1/2} \right\|. \label{eqn:conditions_constraint_satisfaction_3}  
\end{align}
\end{subequations}
The solution to Problem~\ref{Prob:robust_MPC} is the optimal initial condition $\bar{z}^*_{0|k}$ and the sequences:
\begin{align*}
    &\bar{\mathbf{u}}^*_k = \col{(\bar{u}^*_{0|k},\bar{u}^*_{1|k},\dots,\bar{u}^*_{N-1|k})}, \\
    &\bar{\mathbf{z}}^*_k = \col{(\bar{z}^*_{1|k},\bar{z}^*_{2|k},\dots,\bar{z}^*_{N|k})}.
\end{align*}
Next, the geodesic in the lifted state space is found as the minimizer of 
\begin{align}\label{eqn:geodesic_KOOPMAN}
    \gamma^*_{0|k}(s) := \arg \min_{c_k} d_{c_k}(\bar{z}_{0|k}^*,z_k).
\end{align}
Finally, the feedback controller is defined as:
\begin{align}\label{eqn:feedback_contractive}
u_k
&= u_{0|k}^*
+ \underbrace{\int_{0}^{1}
K(\gamma_{0|k}^*(s))
\frac{\partial \gamma_{0|k}^*(s)}{\partial s}\,ds}_{:=\kappa(z_k,\bar{z}_{0|k}^*)}.
\end{align}
The complete feedback controller architecture and its underlying computations are summarized in the block scheme depicted in Figure~\ref{fig:control_architecture}.
\begin{remark}\label{eqn:initial_condition_penalty}
Problem~\ref{Prob:robust_MPC} does not explicitly use the measured state $z_k=\bar{\phi}_{\theta^*}(x_k)$ and therefore operates in a nominal, open-loop manner. The contractive feedback controller~\eqref{eqn:feedback_contractive} keeps the optimized initial-state trajectory $\bar{z}_{0|k}^*$ close to the actual trajectory $z_k$. For improved performance, we recommend adding a regularization term that penalizes their deviation, e.g.,
$\ell_0(\bar{z}_{0|k},z_k) = \gamma_0\|z_{0|k}-z_k\|_2^2$ with $\gamma_0>0$. 
\end{remark}
\begin{remark}
 To compute the geodesic in~\eqref{eqn:geodesic_KOOPMAN} and the feedback control law \eqref{eqn:feedback_contractive}, we employ the Chebyshev pseudospectral method \cite{gong2009chebyshev}. Furthermore, Section~\ref{sec:computation} presents an efficient sequential quadratic programming (SQP) algorithm based on the Chebyshev pseudospectral discretization for solving the geodesic problem~\eqref{eqn:geodesic_KOOPMAN} efficiently online.
\end{remark}

For the reference points $r^z\in\mathbb{Z}$ and $r^u\in\mathbb{U}$, we use the following quadratic stage and terminal costs in Problem~\ref{Prob:robust_MPC}:
\begin{align}
\ell_N(\bar{z}_{N|k})
&:= (\bar{z}_{N|k}-r_{N|k})^\top
P(\bar{z}_{N|k}-r_{N|k}),\\
\ell(\bar{z}_{i|k},r^z,\bar{u}_{i|k},r^u)
&:= (\bar{z}_{i|k}-r^z)^\top
Q(\bar{z}_{i|k}-r^z) \nonumber\\
&\quad+
(\bar{u}_{i|k}-r^u)^\top
R(\bar{u}_{i|k}-r^u).
\end{align}
We assume that $P\succ0$, $Q\succeq0$, and $R\succ0$. Given a desired state reference $r^x$, the corresponding lifted state and steady-state input references are defined as
\begin{align}
    r^z := \bar{\phi}_{\theta^*}(r^x), 
    \qquad
    r^z = A^*r^z + \bar{B}^*(r^z)r^u.
\end{align}
The second relation requires $(r^z,r^u)$ to be an equilibrium pair of the nominal bilinear Koopman model; given $r^z$, a suitable steady-state input $r^u$ can be computed by (approximately) solving the corresponding least-squares problem.

The assumptions concerning the terminal set $\mathbb{Z}_T$ and terminal cost $P$ are presented next.
\begin{assumption}\label{assm:terminal_set}
Given reference points \(r^z\in\bar{\mathbb{Z}}\) and \(r^u\in\bar{\mathbb{U}}\), there exist a stabilizing linear state-feedback controller
\[
u_k=r^u+\bar{K}(z_k-r^z),
\]
where $\bar{K}\in\mathbb{R}^{n_u\times n_z}$, a symmetric positive definite matrix \(P\succ0\), and a compact terminal set
\(\mathbb{Z}_T\subseteq\bar{\mathbb{Z}}\) containing \(r^z\) in its interior such that, for all
\(\rho\in\mathbb{Z}_T\),
\begin{align*}
(A^*+\bar{B}^*(\rho)\bar{K})(\mathbb{Z}_T-r^z)
&\subseteq
(\mathbb{Z}_T-r^z),\\
\bar{K}(\bar{\mathbb{Z}}_T-r^z)
&\subseteq
(\bar{\mathbb{U}}-r^u),\\
(A^*+\bar{B}^*(\rho)\bar{K})^\top P(A^*+\bar{B}^*(\rho)\bar{K})-P
&\preceq
-Q-\bar{K}^\top R\bar{K}.
\end{align*}
\end{assumption}

\begin{theorem}[Constraint satisfaction]
Assume at time instant $k=0$ that Problem~\ref{Prob:robust_MPC} was solved with $z^*_{0|0} = z_0$. Further, Let assumption~\ref{assumption:1} hold and assume that Problem~\ref{Prob:robust_MPC} is feasible at time instant \(k\), where
\begin{align*}
\bar{\mathbf{u}}^f_k
&=
\{\bar{u}^f_{0|k},\bar{u}^f_{1|k},\dots,\bar{u}^f_{N-1|k}\}, \quad \bar{\mathbf{z}}^f_k 
=
\{\bar{z}^f_{0|k},\bar{z}^f_{1|k},\dots,\bar{z}^f_{N|k}\},
\end{align*}
are arbitrary feasible input and state sequences. Let
\begin{align*}
\mathbf{u}^f_k = \{u^f_{0|k},u^f_{1|k},\dots,u^f_{N-1|k}\}, 
\quad \mathbf{z}^f_k = \{z^f_{0|k},z^f_{1|k},\dots,z^f_{N|k}\},
\end{align*}
denote the corresponding closed-loop input and state sequence obtained by
applying the contraction-based feedback law \eqref{eqn:feedback_contractive}, i.e., 
\begin{align}
    u^f_{i|k}
=
\bar{u}^f_{i|k}
+
\kappa(z^f_{i|k},\bar{z}^f_{i|k}),
\qquad i=0,\dots,N-1,
\end{align}
to the system \eqref{eqn:system}, with \(z^f_{0|k}=\bar{\phi}_{\theta^*}(x_k)\). Then the closed-loop trajectory satisfies
$z^f_{i|k}\in\mathbb{Z}$, $u^f_{i|k}\in\mathbb{U}$,
for  $i=0,\dots,N-1$ and $z^f_{N|k}\in\mathbb{Z}$.
\end{theorem}
\begin{proof}
Since Problem~\ref{Prob:robust_MPC} is feasible, then every element of $\bar{\mathbf{u}}^f_k$ and $\bar{\mathbf{z}}^f_k$ belongs to the corresponding admissible sets, i.e.,
$$
\bar{u}^f_{i|k}\in\bar{\mathbb{U}},\quad i=0,\ldots,N-1,
\qquad
\bar{z}^f_{i|k}\in\bar{\mathbb{Z}},\quad i=0,\ldots,N.
$$
    By Proposition~5 in \cite{sasfi2023robust}, we have that for any $\bar{z}^f_{i|k}, z^f_{i|k}\in\mathbb{R}^{n_z}$ and $\bar{u}^f_{i|k}\in\mathbb{R}^{n_u}$, satisfying
\begin{align*}
    &h^u(\bar{u}^f_{i|k})
+
c^u d_{\gamma_{i|k}}(\bar{z}^f_{i|k},z^f_{i|k})
\leq 0, \\
&h^z(\bar{z}^f_{i|k})
+
c^z d_{\gamma_{i|k}}(\bar{z}^f_{i|k},z^f_{i|k})
\leq 0, 
\end{align*}
that 
$v_{i|k}(s)\in\mathbb{U}$ and $\gamma_{i|k}(s)\in\mathbb{Z}$ for all $s\in[0,1]$, where 
$$
v_{i|k}(s) = \bar{u}^f_{i|k} + \int_0^s K(\gamma_{i|k}(s)\frac{\partial \gamma_{i|k}(s)}{\partial s}ds, 
$$
and $\gamma_{i|k}(s)$ is the geodesic between $\gamma_{i|k}(1) =z ^f_{i|k}$ and $\gamma_{i|k}(0) =\bar{z}^f_{i|k}$ and $v_{i|k}(1) = u^f_{i|k}$ and $v_{i|k}(0) =\bar{u}^f_{i|k}$. To guarantee constraint satisfaction of every element of $\mathbf{u}^f_k$ and $\mathbf{z}^f_k$, i.e.,
$$
u^f_{i|k}\in\mathbb{U},\quad i=0,\ldots,N-1,
\qquad
z^f_{i|k}\in\mathbb{Z},\quad i=0,\ldots,N,
$$
it is sufficient that 
$$
d_{\gamma_k}(\bar{z}^f_{i|k},z^f_{i|k}) \leq \delta_{\text{max}}, \quad i = 0,\dots,N,
$$    
since this imposes a more strict constraint tightening. Recall that Theorem~\ref{thrm:robust_contraction} provides that under Assumption~\ref{assumption:1} the robust contraction property holds, i.e., inequality \eqref{eqn:robust_property} holds Therefore, we have that 
\begin{align*}
    E_{\gamma_{i|k}}(\bar{z}^f_{i|k}, z^f_{i|k}) \leq (1-\alpha)^{k+i} E_{\gamma_{0|0}}(\bar{z}^*_{0|0}, z_0) \\
    + \sum_{i=0}^{k+i-1} (1-\alpha)^i\nu \bar{\delta^2}.
\end{align*}
This inequality holds provided that the sequence $\bar{z}^*_{0|0},\dots,\bar{z}^*_{0|k},\dots,\bar{z}^*_{N|k}$ is a trajectory of the Koopman model \eqref{eqn:Koopman_approximation_true}, which is guaranteed by constraint \eqref{eqn:robust_MPC_c3} and the assumption $0\in\mathbb{W}$. Furthermore, it holds that:
 \begin{align*}
     d_{\gamma_{i|k}}(\bar{z}^f_{i|k}, z^f_{i|k}) \leq \sqrt{E_{\gamma_{i|k}}(\bar{z}^f_{i|k}, z^f_{i|k})},
 \end{align*}
We assumed that $z^*_{0|0}=z_0$, therefore $E_{\gamma_{0|0}}(\bar{z}^*_{0|0}, z_0) = 0$ and 
\begin{align*}
d_{\gamma_{i|k}}\left(\bar{z}^f_{i|k},z^f_{i|k}\right)
&\leq
\sqrt{\nu}\bar{\delta}
\sum_{j=0}^{k+i}(1-\alpha)^{j/2} \\
&\leq
\sqrt{\nu}\bar{\delta}
\sum_{j=0}^{\infty}(1-\alpha)^{j/2} \\
&=
\frac{\sqrt{\nu}\bar{\delta}}
{1-\sqrt{1-\alpha}}
=:\delta_{\mathrm{max}},
\end{align*}
where the second inequality follows from $(1-\alpha)^{j/2}\geq0$ for all $j\geq0$, and the equality follows from the geometric-series formula, provided that $0<\alpha\leq1$. Lastly, since $\mathbb{Z}_T\subseteq \bar{\mathbb{Z}}$ it follows that $z^f_{N|k}\in\mathbb{Z}$.
\end{proof}

\begin{remark}\label{eqn:initial_condition_k0}
To obtain the worst-case bound $\delta_{\mathrm{max}}$ for the Riemannian distance, we require $z^*_{0|0}=z_0$. This equality is enforced as an additional constraint for $k=0$. For $k>0$, the optimized initial condition is allowed to deviate from the measured state, which is necessary to preserve recursive feasibility.
\end{remark}

\subsection{Recursive Feasibility and Stability Analysis}
\begin{theorem}[Recursive feasibility]
\label{thrm:recursive_feasibility}
Suppose that Assumptions~\ref{assumption:1} and~\ref{assm:terminal_set} hold and that Problem~\ref{Prob:robust_MPC} is feasible at time instant \(k\). Then, Problem~\ref{Prob:robust_MPC} is recursively feasible. In particular, it is feasible at time instant \(k+1\).
\end{theorem}
\begin{proof}
    At time instant $k$ consider the nominal feasible trajectories:
    \begin{align*}
        \bar{\mathbf{u}}^f_k &= \col{(\bar{u}_{0|k}^f,\bar{u}_{1|k}^f,\dots,\bar{u}_{N-2|k}^f,\bar{u}_{N-1|k}^f)}, \\
        \bar{\mathbf{z}}^f_k &= \col{(\bar{z}_{0|k}^f,\bar{z}_{1|k}^f,\dots,\bar{z}_{N-1|k}^f,\bar{z}_{N|k}^f)}, \\
        \boldsymbol{\rho}^f_k &= \col{(\bar{z}_{0|k}^f,\bar{z}_{1|k}^f,\dots,\bar{z}_{N-2|k}^f,\bar{z}_{N-1|k}^f)}. 
    \end{align*}
    Next, at time instant $k+1$, consider the candidate trajectories:
    \begin{align*}
        \bar{\mathbf{u}}^f_{k+1} &= \col{(\bar{u}_{1|k}^f,\bar{u}_{2|k}^f,\dots,\bar{u}_{N-1|k}^f,r^u+K(\bar{z}_{N|k}^f-r^z))}, \\
        \bar{\mathbf{z}}^f_{k+1} &= \col{(\bar{z}_{1|k}^f,\bar{z}_{2|k}^f,\dots,\bar{z}_{N|k}^f,(A+B(\bar{z}_{N|k}^f)K)\bar{z}_{N|k}^f)}, \\
        \boldsymbol{\rho}^f_{k+1} &= \col{(\bar{z}_{1|k}^f,\bar{z}_{2|k}^f,\dots,\bar{z}_{N-1|k}^f,\bar{z}_{N|k}^f)}. 
    \end{align*}
   Under the standing assumptions and Assumption~\ref{assm:terminal_set}, the candidate trajectory is feasible. Indeed, the first $N-1$ elements of the shifted input and state sequences satisfy the tightened state and input constraints because they are feasible at time instant $k$; the appended terminal state and input are admissible and remain in $\mathbb{Z}_T\subseteq\bar{\mathbb{Z}}$ by the invariance and admissibility conditions of Assumption~\ref{assm:terminal_set}; and the initial-condition constraint \eqref{eqn:robust_MPC_c3} is satisfied at time instant $k+1$ by the candidate initial condition $\bar{z}^f_{0|k+1}=\bar{z}^f_{1|k}$, since then $\|z_{0|k+1}-z^*_{1|k}\|_2 = 0\leq\bar{w}$ whenever the feasible solution at time $k$ is taken as the optimal one. Hence, the candidate sequences constitute a feasible solution of Problem~\ref{Prob:robust_MPC} at time instant $k+1$, which completes the proof.
\end{proof} 

Next, let $\mathbb{Z}_f(N)$ denote the set of feasible initial states, i.e., $\bar{z}_{0|k}$, which are positive invariant by Theorem~\ref{thrm:recursive_feasibility} for the closed-loop system \eqref{eqn:system}--\eqref{eqn:feedback_contractive}. Next, define the optimal value function:
\begin{align*}
    V(\bar{z}^*_{0|k},z_k,r^u,r^z) := \ell_N(\bar{z}_{N|k}^*,r^z) + \sum_{i=0}^{N-1} \ell(\bar{z}^*_{i|k},r^z,\bar{u}^*_{i|k},r^u).
\end{align*}
Before stating the main closed-loop stability result, we recall the notion of ISS for discrete-time systems \cite{JIANG2001857}.
\begin{definition}[Input-to-state stability]\label{def:iss}
Consider a discrete-time system $\chi_{k+1} = F(\chi_k, e_k)$ with state $\chi_k\in\mathbb{R}^{n}$, input $e_k\in\mathbb{R}^{n_e}$, and let $\chi_e$ be an equilibrium of the zero-input dynamics, i.e., $\chi_e = F(\chi_e, 0)$. The system is called input-to-state stable (ISS) with respect to $e$ in a set $\mathbb{X}_0$ if there exist $\beta\in\mathcal{KL}$ and $\gamma\in\mathcal{K}_{\infty}$ such that, for all $\chi_0\in\mathbb{X}_0$ and all bounded input sequences $\{e_k\}_{k\in\mathbb{N}}$, the corresponding trajectories satisfy
$$
\|\chi_k - \chi_e\|_2 \leq \beta(\|\chi_0 - \chi_e\|_2, k) + \gamma(\|\mathbf{e}_{[0,k-1]}\|_{\infty}), \quad \forall k\in\mathbb{N},
$$
where $\mathbf{e}_{[0,k-1]}:=\col(e_0,\dots,e_{k-1})$.
\end{definition}

\begin{theorem}[Input-to-state stability]\label{thrm:practical_stability}
For the closed--loop system \eqref{eqn:system}--\eqref{eqn:feedback_contractive}, suppose that Assumption~\ref{assumption:1}, Assumption~\ref{assm:terminal_set} and the assumptions of Theorem~\ref{thrm:recursive_feasibility} hold for the reference signals $r^z := \bar{\phi}(r^x)$ and $r^u$, and let Problem~\ref{Prob:robust_MPC} be initialized with $z^*_{0|0}=z_0$ at $k=0$ (cf.\ Remark~\ref{eqn:initial_condition_k0}). Then, for all $\bar{z}_{0|k}^*\in\mathbb{Z}_f(N)$ it holds that:
\begin{enumerate}
    \item [(i)] the trajectory of initial conditions is asymptotically stable, i.e.,  
    $$
    \lim_{k\rightarrow \infty }\|z^*_{0|k} - r^z\| = 0,
    $$
    \item[(ii)] the closed-loop system \eqref{eqn:system}--\eqref{eqn:feedback_contractive} is ISS with respect to the disturbance mismatch $e_k := w_k^0 - w_k^s$, i.e., there exist $\tilde{\beta}\in\mathcal{KL}$ and $\tilde{\gamma}\in\mathcal{K}_{\infty}$ such that, for all $k\in\mathbb{N}$,
    $$
    \|x_k - r^x\|_2 \leq \tilde{\beta}(\|x_0 - r^x\|_2, k) + \tilde{\gamma}(\|\mathbf{e}_{[0,k-1]}\|_{\infty}),
    $$
    where $w_k^0, w_k^s\in\mathbb{W}$ denote the disturbance realizations associated with the optimized initial-condition trajectory and the measured lifted trajectory, respectively.
\end{enumerate}  
\end{theorem}
\begin{proof}
    Since $z_{0|k}^*\in\mathbb{Z}_f(N)$, by Theorem~\ref{thrm:recursive_feasibility} we have that Problem~\ref{Prob:robust_MPC} remains feasible. Consider the feasible initial state $\bar{z}^f_{0|k+1} = \bar{z}^*_{1|k}$ and the feasible scheduling, state and input sequences:
    \begin{align*}
        &\boldsymbol{\rho}^f_{k+1} = \{z_{1|k}^*,\dots,z_{N-1|k}^*,z_{N|k}^*\}, \\
        &\bar{\mathbf{u}}^f_{k+1} = \{u_{1|k}^*,\dots,u_{N-1|k}^*,r^u + K(z_{N|k}^*-r^z)\}, \\
        &\bar{\mathbf{z}}^f_{k+1} = \{z_{2|k}^*,\dots,z_{N|k}^*,(A+\bar{B}(z_{N|k}^*)K)z_{N|k}^*\}.
    \end{align*}
    By optimality of the cost it holds that:
    \begin{align*}
       & V(\bar{z}^*_{0|k+1},r^u,r^z) - V(\bar{z}^*_{0|k},r^u,r^z) \\
       & \leq  V(\bar{z}^*_{1|k},r^u,r^z) - V(\bar{z}^*_{0|k},r^u,r^z) \\
       & = (z_{N|k}^*-r^z)^\top A_{\text{cl}}(z_{N|k}^*)^\top PA_{\text{cl}}(z_{N|k}^*)(z_{N|k}^*-r^z)  \\
       & - (z_{N|k}^*-r^z)^\top P (z_{N|k}^*-r^z) \\
       & + (z_{N|k}^*-r^z)^\top (Q + \bar{K}^\top R \bar{K}) (z_{N|k}^*-r^z) \\
       & - (z_{0|k}^*-r^z)^\top Q  (z_{0|k}^*-r^z) - (u_{0|k}^*-r^u)^\top R  (u_{0|k}^*-r^u) \\
       & \leq - (z_{0|k}^*-r^z)^\top Q  (z_{0|k}^*-r^z),
    \end{align*}
    where $A_{\text{cl}}(z_{N|k}^*) := A^*+\bar{B}^*(z_{N|k}^*)\bar{K}$. Then by exploiting standard MPC stability analysis techniques, suitable positive definite lower and upper bounds on $V(\bar{z}^*_{0|k},r^u,r^z)$ can be established, i.e., 
    $$
    \underline{\alpha}(\|\bar{z}_{0|k}^*-r^z\|) \leq V(\bar{z}^*_{0|k},r^u,r^z) \leq \bar{\alpha}(\|\bar{z}_{0|k}^*-r^z\|),
    $$ 
    where $\underline{\alpha}(s) = \lambda_{\text{min}}(Q)s$. For the computation of $\bar{\alpha}(s)$ we refer to, e.g., \cite{MAYNE2000789}. This further yields the desired property (i) via standard Lyapunov arguments. This proves that $\lim_{k\rightarrow \infty }\|z^*_{0|k} - r^z\| = 0$.

    To prove claim (ii) note that the trajectory consisting of the initial conditions, i.e., $\{\bar{z}_{0|k}^*, \bar{z}_{0|k+1}^*, \bar{z}_{0|k+2}^*, \dots \}$ is a valid trajectory of \eqref{eqn:Koopman_approximation_true} due to constraint \eqref{eqn:robust_MPC_c3}. Hence, there exist a disturbance sequence $w^0_k\in\mathbb{W}$, such that 
    $$
    \bar{z}^*_{0|k+1} = f_{K}(\bar{z}^*_{0|k}, u^*_{0|k},w_k^0).
    $$
    Similarly, there exists a disturbance sequence $w_k^s\in\mathbb{W}$ such that 
    $$
    z_{k+1} = f_{K}(z_k, u_k,w_k^s).
    $$
    By Lemma~\ref{lemma:1} we have that 
    \begin{align*}
        \|\bar{z}^*_{0|k+n} - z_{k+n}\|_2 \leq &\sqrt{\frac{\alpha_2}{\alpha_1}}(1-\alpha)^{n/2}\|\bar{z}^*_{0|k} - z_{k}\|_2 \\
       & \hspace{-20mm}+ \sqrt{\frac{\nu}{\alpha_1}}\sum_{i=0}^{n-1}(1-\alpha)^{i/2} \|w^0_{k+n-1-i}-w^s_{k+n-1-i}\|_2.
    \end{align*}
    Next, define 
     \begin{align*}
        e_k := w^0_{k}-w^s_{k},
    \end{align*}    
    and let $\mathbf{e}_{[k,k+n-1]}:= \col(e_k,e_{k+1},\dots,e_{k+n-1})$. Next, we upper bound by $\|\mathbf{e}_{[k,k+n-1]}\|_{\infty}\geq \|w^0_{k+n-1-i}-w^s_{k+n-1-i}\|_2$. In this case, the second term becomes a geometric series and we can further upper bound as 
    \begin{align}\label{eqm:iss_bound}
        \|\bar{z}^*_{0|k+n} - z_{k+n}\|_2 & \leq \beta(\|\bar{z}^*_{0|k} - z_k\|_2,n) + \gamma(\|\mathbf{e}_{[k,k+n-1]}\|_{\infty}),
    \end{align}
    where 
    \begin{align*}
       &\beta(\|\bar{z}^*_{0|k} - z_k\|_2,n) := \sqrt{\frac{\alpha_2}{\alpha_1}}(1-\alpha)^{n/2}\|\bar{z}^*_{0|k} - z_k\|_2 ,\\
       &\gamma(\|\mathbf{e}_{[k,k+n-1]}\|_{\infty}) := \sqrt{\frac{\nu}{\alpha_1}}\frac{1}{1-\sqrt{1-\alpha}} \|\mathbf{e}_{[k,k+n-1]}\|_{\infty}, 
    \end{align*}
    where $\beta\in\mathcal{KL}$ since $\alpha\in(0,1)$ and $\gamma\in\mathcal{K}_{\infty}$. Setting $k=0$ in \eqref{eqm:iss_bound} and using the initialization $z^*_{0|0}=z_0$ (cf.\ Remark~\ref{eqn:initial_condition_k0}), which implies $\beta(\|\bar{z}^*_{0|0}-z_0\|_2,n)=\beta(0,n)=0$, yields
    $$
    \|\bar{z}^*_{0|n} - z_{n}\|_2 \leq \gamma(\|\mathbf{e}_{[0,n-1]}\|_{\infty}), \quad \forall n\in\mathbb{N}.
    $$
    Furthermore, the Lyapunov inequalities established in the proof of claim (i) imply, by standard comparison arguments for discrete-time systems (see, e.g., \cite{JIANG2001857}), the existence of a function $\beta_z\in\mathcal{KL}$ such that
    $$
    \|\bar{z}^*_{0|k} - r^z\|_2 \leq \beta_z(\|\bar{z}^*_{0|0} - r^z\|_2,k), \quad \forall k\in\mathbb{N}.
    $$
    Combining the two bounds via the triangle inequality gives
    \begin{align*}
    \|z_k - r^z\|_2 &\leq \|z_k - \bar{z}^*_{0|k}\|_2 + \|\bar{z}^*_{0|k} - r^z\|_2 \\
    &\leq \beta_z(\|z_0 - r^z\|_2,k) + \gamma(\|\mathbf{e}_{[0,k-1]}\|_{\infty}),
    \end{align*}
    where we used $\bar{z}^*_{0|0}=z_0$. Since the original state is contained in the lifted state, it holds that $\|x_k - r^x\|_2 \leq \|z_k - r^z\|_2$. Moreover, since $\bar{\phi}_{\theta^*}$ is Lipschitz continuous on the compact set $\mathbb{X}$, there exists $L_{\phi}>0$ such that $\|z_0 - r^z\|_2 = \|\bar{\phi}_{\theta^*}(x_0)-\bar{\phi}_{\theta^*}(r^x)\|_2 \leq L_{\phi}\|x_0 - r^x\|_2$. Hence, the ISS estimate of claim (ii) holds with $\tilde{\beta}(s,k):=\beta_z(L_{\phi}s,k)\in\mathcal{KL}$ and $\tilde{\gamma}:=\gamma\in\mathcal{K}_{\infty}$, which proves claim (ii).
\end{proof}

\begin{remark}[Implications of ISS]\label{rem:iss_implications}
Theorem~\ref{thrm:practical_stability} establishes ISS of the closed-loop system with respect to the model mismatch $e_k$, which is a significantly stronger property than input-to-state \emph{practical} stability (ISpS) or practical asymptotic stability, where a nonzero offset can persist even for vanishing disturbances. In particular, since $\|e_k\|_2\leq\bar{\delta}$ for all $k\in\mathbb{N}$ and $\tilde{\gamma}$ is linear, ISS directly implies the asymptotic gain property \cite{JIANG2001857}
\begin{align*}
\limsup_{k\rightarrow\infty}\|x_k - r^x\|_2 &\leq \tilde{\gamma}\Big(\limsup_{k\rightarrow\infty}\|e_k\|_2\Big)\\ &\leq \sqrt{\frac{\nu}{\alpha_1}}\frac{\bar{\delta}}{1-\sqrt{1-\alpha}},
\end{align*}
i.e., an ultimate bound proportional to the asymptotic magnitude of the disturbance mismatch, as well as the convergence property: if $\lim_{k\rightarrow\infty}\|e_k\|_2=0$, then $\lim_{k\rightarrow\infty}\|x_k-r^x\|_2=0$. This should be contrasted with tube-based MPC schemes that certify convergence of the true state to a robust positively invariant set, see, e.g., \cite{mayne2005robust}. The same observation applies to contraction-based tube MPC: the schemes in \cite{ZhaoGridding,sasfi2023robust,YangCTubeDT2024} certify containment of the true trajectory in a tube around the nominal one, or convergence to a neighborhood of the target, although the underlying (robust) contraction certificates are inherently of incremental ISS type (cf.\ Remark~\ref{rem:delta_iss}). Theorem~\ref{thrm:practical_stability} makes the stronger closed-loop implication explicit, by combining the contraction-based tracking bound \eqref{eqm:iss_bound} with the Lyapunov decrease of the nominal value function into a single ISS estimate with respect to the reference. Similarly, robust and tube MPC schemes based on linear Koopman realizations typically establish ISpS-type or boundedness and set-convergence guarantees, see, e.g., \cite{zhang2022robust}, while the concurrent bilinear scheme \cite{higuchi2026bilinear} certifies convergence to a neighborhood of the target. Since linear Koopman models are recovered as the special case $B_i=\mathbf{0}$, $i=1,\dots,n_u$, of the bilinear model \eqref{eqn:Koopman_approximation_true}, Theorem~\ref{thrm:practical_stability} also provides, as a by-product, ISS guarantees for tube MPC based on linear Koopman realizations, thereby strengthening the closed-loop guarantees available in the literature.
\end{remark}

\begin{remark}[Relation to proportional error bounds]\label{rem:proportional}
An alternative route to exact asymptotic stability of a controlled equilibrium is to require \emph{proportional} error bounds, i.e., bounds that are linear in the norms of the (lifted) state and input and hence vanish at the origin. Such bounds can be certified for bilinear and kernel-based EDMD surrogates \cite{strasser2024safedmd,strasser2025kernel} and yield asymptotic stability of the origin for Koopman-based MPC with and without terminal conditions \cite{schimperna2026terminal,schimperna2025asymptotic}, as well as for Koopman-based robust feedback designs \cite{strasser2025feedback}. These guarantees are anchored at the origin, whereas Theorem~\ref{thrm:practical_stability} holds for arbitrary references $r^x$ under hard state and input constraints. The two routes are complementary within the proposed framework: if a proportional bound holds in a neighborhood of $(r^z,r^u)$, then the mismatch $e_k$ vanishes along converging closed-loop trajectories and the convergence property of ISS recovers asymptotic stability of the reference, subject to a suitable small-gain condition. Integrating proportional and kernel-based deterministic bounds within the developed RCCM tube framework is therefore a promising research direction.
\end{remark}

\section{Computational Framework for Robust MPC}
\label{sec:computation}
The online implementation of the proposed scheme requires solving Problem~\ref{Prob:robust_MPC}, computing the geodesic \eqref{eqn:geodesic_KOOPMAN} and evaluating the contractive feedback \eqref{eqn:feedback_contractive} at every time instant. This section develops a computational framework that reduces all of these tasks to (sequences of) quadratic programs, by exploiting the LPV structure of the bilinear Koopman model and a Chebyshev pseudospectral discretization of the geodesic.

Chebyshev polynomials can be used to accurately approximate functions $f:[0,1]\rightarrow \mathbb{R}$ and to compute integrals with high accuracy, or even analytically in the case of polynomial functions, see \cite{gong2009chebyshev,novelinkova2011comparison}. The $j^{\text{th}}$ Chebyshev basis function is defined by $T_j(s) = \cos(j \cos^{-1}(2s-1))$, or equivalently
\begin{align*}
T_0(s) &= 1\\
T_1(s) &= 2s-1\\
T_2(s) &= 8s^2 - 8s + 1\\
\vdots
\end{align*}
furthermore, the derivative of the basis functions are given by 
\begin{align*}
\frac{dT_0(s)}{ds} &= 0\\
\frac{dT_1(s)}{ds} &= 2\\
\frac{dT_2(s)}{ds} &= 16s - 8\\
\vdots
\end{align*}
These basis functions can be used to approximate a function $f$ and its derivative as
\begin{align*}
    f(s) \approx \sum_{i=0}^{N_{\gamma}} c_i T_i(s), \quad \frac{d f(s)}{ds}  \approx \sum_{i=0}^{N_{\gamma}} c_i \frac{d T_i(s)}{ds},
\end{align*}
where the coefficients have to be optimized. What makes these basis functions so interesting for computing geodesics, is that they can also approximate integrals very accurately and even analytically when ${N_{\gamma}}$ is chosen large enough and the functions is polynomial. An integral of a function $f:[0,1]\rightarrow \mathbb{R}$ can be approximated as
\begin{align}\label{eqn:integral}
    \int_{0}^{1} f(s) dx \approx \sum_{k=0}^{N_w} w_k f(s_k),
\end{align}
where  the points $s_k$ are the Chebyshev-Gauss-Lobatto nodes \cite{gong2009chebyshev, novelinkova2011comparison} defined by
\begin{align*}
    s_k = \frac{1}{2} \left( 1 + \cos\!\left(\frac{\pi (N_w-k)}{N_w}\right) \right), 
\quad k = 0,1,\dots,N_w,
\end{align*}
The coefficients $w_k$ are the quadrature weights \cite{gong2009chebyshev} are given by 
\begin{align*}
    w_k = w_{N_w-k} = \frac{4}{N_w} \sum_{j=0}^{\lfloor N_w/2 \rfloor'} \frac{1}{1-4j^2} \cos\left( \frac{2\pi j k}{N_w}\right), \\
    \quad s = 1,2,\dots,\lfloor N_w/2 \rfloor,
\end{align*}
where the $'$ after the sum means the first and last elements of the sum are multiplied by $0.5$ and the first and last weights are given by 
\begin{align}
    w_0 = w_{N_w} = \begin{cases}
        \frac{1}{N_w^2-1}, \quad &N_w \text{ is even}, \\
        \frac{1}{N_w^2}, \quad &N_w \text{ is not even}.
    \end{cases}
\end{align}
Next, we parameterize the  geodesic and its derivative as 
$$
\gamma(s) \approx   \underbrace{\begin{bmatrix}
    c_1^1 & \dots & c_1^{N_{\gamma}} \\ 
    \vdots & \vdots & \vdots \\
    c_{n_z}^1 & \dots & c_{n_z}^{N_{\gamma}}
\end{bmatrix}}_{\mathbf{C}}\underbrace{\begin{bmatrix}
    T_1(s) \\
    \vdots \\
    T_{N_{\gamma}}(s)
\end{bmatrix}}_{\mathbf{T}(s)},
$$
$$
\frac{d\gamma(s)}{ds} \approx \underbrace{\begin{bmatrix}
    c_1^1 & \dots & c_1^{N_{\gamma}} \\ 
    \vdots & \vdots & \vdots \\
    c_{n_z}^1 & \dots & c_{n_z}^{N_{\gamma}}
\end{bmatrix}}_{\mathbf{C}}\underbrace{\begin{bmatrix}
    \frac{dT_1(s)}{ds} \\
    \vdots \\
    \frac{dT_{N_{\gamma}}(s)}{ds}
\end{bmatrix}}_{\mathbf{T}_d(s)},
$$
where ${N_{\gamma}}\in\mathbb{N}$ is the number of basis functions that will be used for approximating the geodesic. Next, we define the optimization problem for computing the geodesic.
\begin{problem}[geodesic computation]\label{problem:geodesic}
    \begin{subequations}
    \begin{align}
            \mathbf{C}^* = \arg \min_{\mathbf{C}} & \sum_{\ell=0}^{N_w}\mathbf{T}_d(s_\ell)^\top \mathbf{C}^\top W_{\theta_w^*}^{-1}(\mathbf{C} \mathbf{T}(s_\ell)) \mathbf{C} \mathbf{T}_d(s_\ell) w_\ell \\ 
            &\quad \quad \text{subject to:} \nonumber \\
            &\mathbf{C}\mathbf{T}(0) = \bar{z}^*_{0|k} \\
            &\mathbf{C}\mathbf{T}(1) = \bar{\phi}_{\theta^*}(x_k)
    \end{align}        
    \end{subequations}
\end{problem}
Which is a nonlinear problem with linear equality constraints. The main difficulty in this problem arises from the inversion of the trained metric $W_{\theta_w^*}$. This issue is specific to the discrete-time formulation and stems from the structure of the metric learning problem. To address this computational challenge, we develop an iterative QP approach, presented next. Moreover, the contraction based control law can be approximated as
\begin{align*}
      u_k &=  \bar{u}^*_{0|k} + \sum_{k=1}^{N_w} L_{\theta_l^*}(\mathbf{C}^*\mathbf{T}(s_k)) W^{-1}_{\theta_w^*}(\mathbf{C}^*\mathbf{T}(s_k)) \mathbf{C}^*\mathbf{T}_d(s_k) w_k \\
      &\approx \bar{u}^*_{0|k} + \int_0^1 K(\gamma^*_{0|k}(s))\frac{\partial \gamma^*_{0|k}(s)}{\partial s} ds. \label{control_input}
\end{align*}
It is not straightforward to solve Problem~\ref{problem:geodesic}, since $ W_{\theta_w^*}$ is generally a nonlinear matrix function and the cost function requires the inverse of this matrix function which is hard to solve. Next, we reformulate the problem as a sequence of quadratic programs. To this end, we exploit a property of the Kronecker product to express the geodesic and its derivative as
\begin{subequations}
    \begin{align}
\mathbf{C}\mathbf{T}(s)
&= (I_n \otimes \mathbf{T}^{\top}(s))\operatorname{vec}(\mathbf{C}), \\
\mathbf{C}\mathbf{T}_d(s)
&= (I_n \otimes \mathbf{T}_d^{\top}(s))\operatorname{vec}(\mathbf{C}).
\end{align}
\end{subequations}
Let $\mathbf{C}_v := \operatorname{vec}(\mathbf{C})$. We solve the geodesic computation problem iteratively. At iteration $j\in\mathbb{N}$, let $\mathbf{C}_v^{(j-1)^*}$ denote the current estimate of $\mathbf{C}_v$. The geodesic computation problem can then be reformulated as a quadratic program around this estimate, as presented next.
\begin{problem}[geodesic computation QP form iteration ($j$)]\label{problem:geodesic_QP}
    \begin{align*}
&\mathbf{C}^{(j)^*}_v :=\arg \min_{\mathbf{C}^{(j)}_{v}}\sum_{\ell=0}^{N_w} (\mathbf{C}^{(j)}_v)^\top V(\mathbf{C}^{(j-1)^*},\ell) \mathbf{C}^{(j)}_{v} \\
&\text{subject to:}\\
& (I_n \otimes \mathbf{T}^\top(0)) \mathbf{C}^{(j)}_{v} = \bar{z}^*_{0|k}, \\
& (I_n \otimes \mathbf{T}^\top(1)) \mathbf{C}^{(j)}_{v} = \bar{\phi}_{\theta^*}(x_k), 
\end{align*}
\end{problem}
\noindent where 
\begin{align*}
&V\!\left(\mathbf{C}^{(j-1)^*},\ell\right)\\
&:= (I_n\otimes\mathbf{T}^{\top}(s_\ell))^{\top}
\left[W\!\left((I_n\otimes\mathbf{T}^{\top}(s_\ell))\mathbf{C}^{(j-1)^*}\right)\right]^{-1}
\nonumber\\
&\qquad \cdot (I_n\otimes\mathbf{T}^{\top}(s_\ell))w_\ell .
\end{align*}
Above, $V\!\left(\mathbf{C}^{(j-1)^*},\ell\right)$ is a positive semidefinite matrix and the problem is therefore a QP. The complete online procedure, which combines the LPV-based MPC iterations with the geodesic QP iterations and implements the control diagram depicted in Figure~\ref{fig:control_architecture}, is summarized in Algorithm~\ref{alg:robust_MPC}. The next section evaluates the complete framework through a numerical example.

\begin{algorithm}[t]
\caption{\emph{Online} contraction-based robust MPC at time instant $k$}
\label{alg:robust_MPC}
\begin{algorithmic}[1]
\Require Trained Koopman model $f^*_K$, lifting map $\bar{\phi}_{\theta^*}$, metric $W_{\theta^*_w}$, control law $L_{\theta^*_l}$, MPC iteration number $N_{\mathrm{mpc}}$, and geodesic iteration number $N_{\mathrm{geo}}$, number of basis functions $N_{\gamma}$, number Chebyshev-Gauss-Lobatto nodes $N_w$, forgetting factors $\alpha_{\rho}\in[0,1]$ and $\alpha_{\gamma}\in[0,1]$.
\State Measure the current state $x_k$ from system~\eqref{eqn:system}.
\State Compute the lifted state $z_k=\bar{\phi}_{\theta^*}(x_k)$.
\If{$k=0$}
    \State Add the constraint $z_{0|0}=z_0$ to Problem~\ref{Prob:robust_MPC}.
    \State Initialize $\boldsymbol{\rho}_0 = \col{(z_0, \dots,z_0)}$.
\EndIf
\If{$k>0$}
\State Initialize the predicted state sequence $\boldsymbol{\rho}_k=\col{(z_{1|k-1}^{*}, \dots,z_{N|k-1}^{*})}$.
\EndIf
\For{$j=1,\ldots,N_{\mathrm{mpc}}$}
    \State Solve Problem~\ref{Prob:robust_MPC} using $\boldsymbol{\rho}_k$.
    \State Obtain the nominal input $\bar{u}_{0|k}^{*}$, initial state $z_{0|k}^{*}$, 
    \State Update $\boldsymbol{\rho}_k = \alpha_{\rho}\col{(z_{0|k}^{*}, \dots,z_{N-1|k}^{*})} + (1-\alpha_{\rho})\boldsymbol{\rho}_k$.
\EndFor
\State Initialize $\mathbf{C}^{(0)}_v$ to form a straight-line path between $z_{0|k}^{*}$ and $z_k$.
\For{$j=1,\ldots,N_{\mathrm{geo}}$}
    \State Solve Problem~\ref{problem:geodesic_QP} using $\mathbf{C}^{(j-1)*}_v$.
    \State Set $\mathbf{C}^{(j)*}_v = \alpha_{c}\mathbf{C}^{(j)*}_v + (1-\alpha_{c})\mathbf{C}^{(j-1)*}_v$.
\EndFor
\State Compute the contractive feedback correction:
\begin{align*}
\hat{\gamma}^*_{0|k}(s) &:= (I_n \otimes \mathbf{T}^{\top}(s)) \mathbf{C}^{(N_{\mathrm{geo}})*}_v, \\
\frac{\partial\hat{\gamma}^*_{0|k}}{\partial s}(s) &:= (I_n \otimes \mathbf{T}_d^{\top}(s)) \mathbf{C}^{(N_{\mathrm{geo}})*}_v, \\
    \Delta u_k &=
    \sum_{\ell=1}^{N_w} L_{\theta_l^*}(\hat{\gamma}^*_{0|k}(s_\ell)) W^{-1}_{\theta_w^*}(\hat{\gamma}^*_{0|k}(s_\ell)) \frac{\partial\hat{\gamma}^*_{0|k}}{\partial s}(s_{\ell}) w_{\ell}.
\end{align*}
\State Apply $u_k=\bar{u}_{0|k}^{*}+\Delta u_k$ to system~\eqref{eqn:system}.
\end{algorithmic}
\end{algorithm}

\section{Numerical Example}
\label{sec:example}
We evaluate the effectiveness of the contraction-based tube MPC framework on a nonlinear pendulum system with a state-dependent input channel. The dynamics are given by the control-affine nonlinear system
\begin{subequations}\label{eqn:pendulum}
    \begin{align}
        \dot{x}_1 &= x_2, \\
        \dot{x}_2 &= -\frac{g}{l}\sin(x_1) - \frac{d}{m l^2}x_2
        + \frac{\cos(x_1)}{m l^2}u .
    \end{align}
\end{subequations}
The system parameters are selected as $m=1$, $l=1$, $g=9.81$, and $d=0.3$. The continuous-time dynamics are discretized using a fourth-order Runge--Kutta integration scheme with sampling time $T_s=1/30$. The input term represents a nonlinear state-dependent forcing term, which can be interpreted as a horizontal force applied to the pendulum rather than a direct torque input. This choice introduces a nonlinear input coupling that increases the difficulty of the system identification problem and provides a suitable benchmark for evaluating the proposed bilinear Koopman model. 

\begin{figure}[t]
\centering
\includegraphics{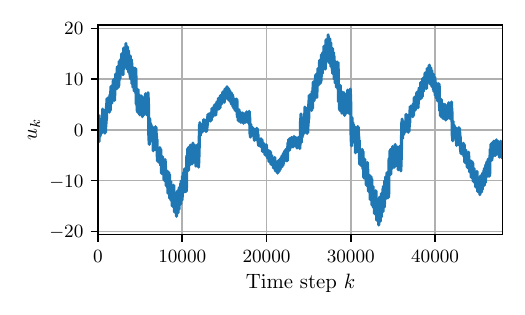}
\caption{Input signal used for model identification.}
\label{fig:input_id}
\end{figure}

\begin{figure}[t]
\centering
\includegraphics{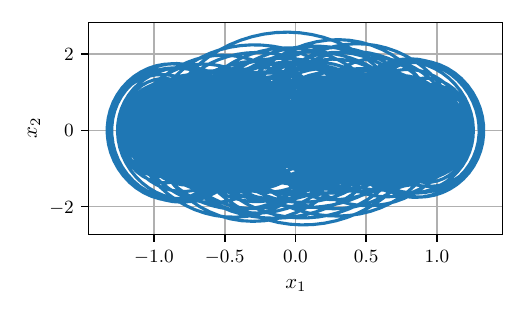}
\caption{State data used for model identification.}
\label{fig:state_id}
\end{figure}

\paragraph{Bilinear Koopman Model Learning}
To identify a model for \eqref{eqn:pendulum}, we perform an identification experiment by exciting the system with a piecewise-constant input signal augmented by a random dithering signal to ensure sufficient excitation and exploration of the state-space. We generate in total $N_s=50000$ data samples. The input and state data that is used for training is shown in Figure~\ref{fig:input_id} and Figure~\ref{fig:state_id}. 
The input signal is chosen in such a way that the state-space is covered as much as possible.

Next, using the collected input-state data, we solve Problem~\ref{prob:model_learning} with $N=30$, which is equal to the prediction horizon that is used later in the MPC problem. While it is not necessary to have these equal, it promotes the minimization of the multi-step prediction error. We used $40000$ epochs for training the models. The training and validation loss over the epochs are displayed in Figure~\ref{fig:loss_model_learning}. The maximum modeling error \eqref{eqn:approximation_error} over the training dataset was equal to $\max_k{\|w_k\|_2} < 0.01$. We therefore assume that the disturbance set is given by $\mathbb{W} := \{w\in\mathbb{R}^{n_z} \vert \|w\|_2 \leq 0.01\}$.

\begin{figure}[t]
\centering
\includegraphics{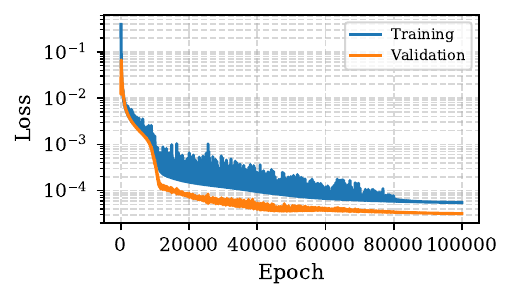}
\caption{Training and validation loss for over epochs for solving Problem~\ref{prob:model_learning}.}
\label{fig:loss_model_learning}
\end{figure}

\begin{figure}[t]
\centering
\includegraphics[
    width=0.9\linewidth,
    trim={0.2cm 0.2cm 0.2cm 1.0cm}, 
    clip
]{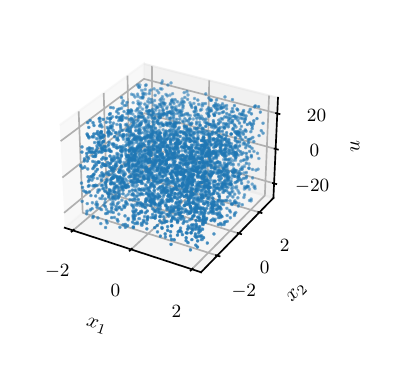}
\caption{State $\hat{x}_i$ and input $\hat{u}_i$ data points.}
\label{fig:metric_data}
\end{figure}

\paragraph{Contraction Metric and Controller Learning}
    To train the contraction metric we generate a data set of input, state and disturbance data points $\mathcal{D}_M = \{\hat{x}_k,\hat{u}_k,\hat{w}_k\}_{k=0}^{N_M}$ with $N_M=16,875$. The states and inputs are sampled from the intervals $-2\leq x_1\leq 2$, $-3\leq x_2\leq 3$, $-25\leq u\leq 25$ and $-0.01\leq w \leq 0.01$, Figure~\ref{fig:metric_data} shows the resulting state and input samples. Next, we apply Algorithm~\ref{alg:sparse_sampling} with $N_{\epsilon}=5$ and $\delta_j^{\epsilon} \in \{0, 0.001, 0.0025, 0.005, 0.01\}$, for $j=1,\ldots,5$, and solve Problem~\ref{prob:metric_learning} with the hyperparameters shown in Table~\ref{tab:hyperparameters1}. The metric  $W_{\theta_w}\in\mathcal{NN}$ is parameterized as a feedforward neural network with $2$ hidden layers and $128$ neurons per layer and the controller $L_{\theta_\ell}\in\mathcal{NN}$ is parameterized as a  feedforward neural network with $2$ hidden layers and $64$ neurons per layer. In total we perform $10000$ epochs. For the trained metric and controller to satisfy the contraction inequality, the matrix $M_{i,j}$ in \eqref{eqn:contraction_conbstraint} must be positive definite over the dataset. The percentages of positive eigenvalues for the training and validation datasets are $90.4\%$ and $90.2\%$, respectively. Although these values are reasonably high, they remain below the desired level, as the contraction inequality is enforced over the entire constraint set. Future work will investigate more informed sampling strategies, such as sampling along system trajectories, to improve the coverage of relevant regions of the state space, thereby increasing the percentage of positive eigenvalues and potentially yielding tighter bounds. This means that the trained metric satisfies the contraction inequality \eqref{eqn:contraction_inequality_koopman} almost everywhere. The minimum and maximum eigenvalues of $W_{\theta^*_w}$ over the entire dataset were equal to $\lambda_{\text{max}}= 9.3458$ and $\lambda_{\text{min}} = 2.5510$, leading to a condition number of approximately $3.6$.

\begin{table}[t]
\centering
\caption{Hyperparameters for Problem~\ref{prob:metric_learning}.}
\label{tab:hyperparameters1}

\begin{tabular}{cccccc}
\toprule
\textbf{Hyperparameter} & $\alpha$ & $\nu$ & $\gamma_c$ & $\beta$ & $\varepsilon$ \\
\midrule
\textbf{Value} & 0.1 & 0.5 & 1 & 5 & 10 \\
\bottomrule
\end{tabular}

\end{table}

\paragraph{Constraint tightening computations}
The state and input constraint functions ($h^u$ and $h^z$) in Problem~\ref{Prob:robust_MPC} are defined as:
\begin{align*}
   & h^u(u) = \begin{bmatrix}
        \frac{1}{25} \\
        -\frac{1}{25} 
    \end{bmatrix} u - \begin{bmatrix}
        1 \\1
    \end{bmatrix}, \\ 
    &h^z(z) = \begin{bmatrix}
        \frac{1}{2} & 0 & 0 & \dots & 0\\
        -\frac{1}{2}  & 0 & 0 & \dots & 0 \\ 
       0 &  \frac{1}{3}   & 0 & \dots & 0 \\
        0 &  -\frac{1}{3}  & 0 & \dots & 0 \\ 
    \end{bmatrix} z - \begin{bmatrix}
        1 \\1 \\1 \\1
    \end{bmatrix}. 
\end{align*}
The tightening parameter $\delta_{\text{max}} = 0.2756$ is computed using the parameters from Table~\ref{tab:hyperparameters1} and using $\bar{\delta} = 2\bar{w} = 0.02$. To compute the $c^u$ and $c^z$ we evaluate the formulas \eqref{eqn:constraint_tightening_constant} in a similar way as the dataset for training the metric. Hence, we generate a large set of points as in Figure~\ref{fig:metric_data} and map this grid through the trained observable function $\bar{\phi}_{\theta^*}$. This gave $c^z=\col(1.31,1.31,1.06,1.06)$ and $c^u=\col(0.77,0.77)$. To obtain a terminal set $\mathbb{Z}_T$ and terminal cost $P$ satisfying Assumption~\ref{assm:terminal_set}, we use MPT3 to compute terminal sets for a grid of lifted reference states $\rho_i = \bar{\phi}_{\theta^*}(r_i^x)$, $r_i^x \in \left\{\operatorname{col}(1,0),\operatorname{col}(0.5,0),\operatorname{col}(0,0),\operatorname{col}(-0.5,0)\right\}$. Due to the relatively high dimensionality of the lifted state space, the resulting terminal sets contained more than $400$ constraints. This substantially increased the computational burden of the MPC controller, rendering the simulations computationally impractical. Therefore, we choose $\mathbb{Z}_T=\bar{\mathbb{Z}}$. The computation of less conservative terminal sets and costs for high-dimensional systems is left for future work.

\paragraph{Robust MPC Simulation Results}
Finally, we implement the robust MPC controller using Algorithm~\ref{alg:robust_MPC} with
$N_{\mathrm{geo}}=2$, $N_{\mathrm{mpc}}=2$, $N_w=20$, $N_\gamma=10$,
$\alpha_{\rho}=0.01$, and $\alpha_{\gamma}=0.01$. The geodesic and contractive feedback computations (lines $10$--$14$ of Algorithm~\ref{alg:robust_MPC}) averaged $0.079\mathrm{s}$, while the MPC solve (lines $15$--$21$) averaged $0.023\mathrm{s}$. All methods were implemented in CasADi, with the model predictive control (MPC) problem solved using MOSEK and the geodesic computed using OSQP, on a 13th Gen Intel\textsuperscript{\textregistered} Core\textsuperscript{\texttrademark} i7-1370P CPU at 1.90~GHz. The resulting closed-loop trajectories of $z_k$ and $u_k$ are shown in Figure~\ref{fig:closed_loop}, together with the corresponding MPC trajectories $\bar{z}_{0|k}^*$ and $\bar{u}_{0|k}^*$. As discussed in Remark~\ref{eqn:initial_condition_penalty}, we augment Problem~\ref{Prob:robust_MPC} with the regularization term
\[
\ell_0(\bar{z}_{0|k},z_k)
=
\gamma_0\left\|\bar{z}_{0|k}-z_k\right\|_2^2,
\qquad \gamma_0=1000,
\]
to encourage the MPC trajectory to remain close to the measured system trajectory. The shaded blue regions indicate the tubes used for constraint tightening, defined as
\begin{align*}
    T_{x_1}(k)
    &=
    \left\{
    x\in\mathbb{R}
    \,\middle|\,
    x_1(k)-c_2^z \leq x \leq x_1(k)+c_1^z
    \right\}, \\
    T_{x_2}(k)
    &=
    \left\{
    x\in\mathbb{R}
    \,\middle|\,
    x_2(k)-c_4^z \leq x \leq x_2(k)+c_3^z
    \right\}, \\
    T_u(k)
    &=
    \left\{
    u\in\mathbb{R}
    \,\middle|\,
    u_k-c_2^u \leq u \leq u_k+c_1^u
    \right\}.
\end{align*}

\begin{figure}[t]
\centering
\includegraphics[
    width=\linewidth,
    trim={0.2cm 0.2cm 0.2cm 0.2cm}, 
    clip
]{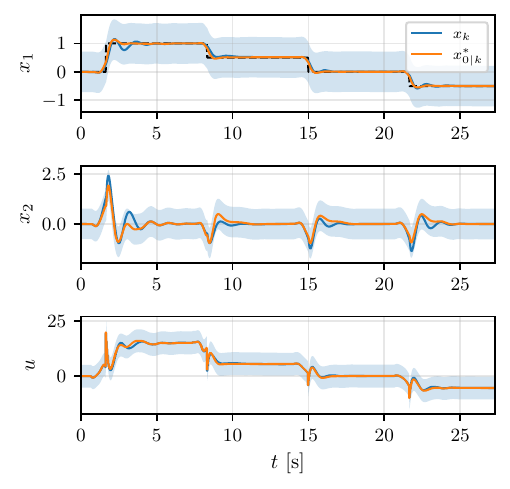}
\caption{Closed-loop state trajectory $x_k$ and optimized initial-state trajectory $\bar{x}^*_{0|k}$, together with the corresponding tube. The input plot shows the optimal MPC input $u^*_{0|k}$ and the applied control input $u_k$.}
\label{fig:closed_loop}
\end{figure}

Figure~\ref{fig:closed_loop} shows that the reference is tracked accurately, with small visible steady-state offset. The results also indicate that the computed tube is relatively conservative, particularly for $x_1$. In contrast, the tube for $x_2$ is comparatively tight, as the closed-loop trajectory approaches its boundary during the rapid reference changes. The conservatism of the tubes is primarily attributed to the data set used for metric learning, shown in Figure~\ref{fig:metric_data}. In the current implementation, the metric is trained using samples from the entire constraint sets $\mathbb{X}$ and $\mathbb{U}$, which can result in overly conservative bounds in regions that are not frequently visited by the closed-loop system.

Future work will therefore investigate more informed sampling strategies, such as sampling along closed-loop trajectories generated by the trained Koopman model. This is enabled by the contraction-based feedback law~\eqref{eqn:feedback_contractive}, which guarantees that the closed-loop system trajectory remains sufficiently close to the nominal MPC trajectory, as established by~\eqref{eqn:2norm_bound}. Consequently, sampling along nominal closed-loop trajectories can focus the metric learning data on dynamically relevant regions of the state-input space while still accounting for the deviations induced by the system dynamics. This is expected to reduce the conservatism of the learned metric and, consequently, yield tighter tubes.

\section{Conclusion}
\label{sec:conclusion}
This paper presented a robust tube MPC framework for nonlinear systems represented by data-driven bilinear Koopman models. Discrete-time robust control contraction metrics were derived for bilinear Koopman models, yielding contraction certificates that explicitly account for the mismatch between the Koopman model and the true dynamics via a disturbance gain. To synthesize such certificates in the high-dimensional lifted state space, where SOS-based methods are intractable, a sparse sampling strategy combined with neural network parameterizations of the metric and controller was proposed. Based on the resulting certificates, a robust MPC scheme with tightened constraints and LPV-based terminal ingredients was formulated, and recursive feasibility, robust constraint satisfaction and input-to-state stability (ISS) of the closed-loop system with respect to the model mismatch were established. By exploiting the LPV structure of the bilinear Koopman model and a Chebyshev pseudospectral discretization of the geodesic, the online computations reduce to a convex MPC problem and a sequence of quadratic programs. The complete framework was validated on a nonlinear pendulum benchmark with a state-dependent input gain.

Future work will focus on informed sampling strategies along closed-loop trajectories to reduce the conservatism of the learned certificates, on the computation of less conservative terminal sets and costs for high-dimensional lifted models, and on probabilistic validation bounds for both the estimated disturbance set and the learned contraction certificates. Moreover, integrating proportional and kernel-based error bounds \cite{schimperna2025asymptotic,strasser2024safedmd,strasser2025kernel} within the developed RCCM tube framework is a promising route to strengthen the established ISS guarantee to asymptotic stability of the reference.

\end{document}

%% file: figures/metric_learning_sparse.tex
\begin{tikzpicture}[
    >=Latex,
    every node/.style={inner sep=0pt}
]

\node (left)
    {\includegraphics[width=0.35\columnwidth]{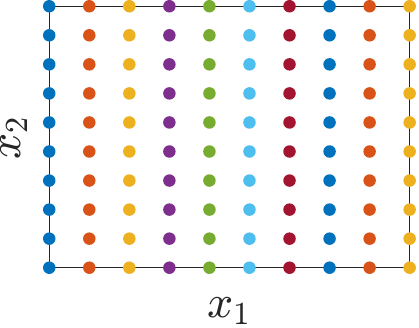}};

\node[right=1.95cm of left] (right)
    {\includegraphics[width=0.4\columnwidth]{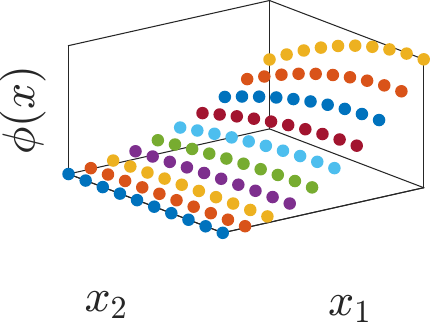}};

\draw[-{Latex[length=1.8mm,width=1mm]}, thick]
    ($(left.east)+(2.5mm,0)$) --
    ($(right.west)+(-2.5mm,0)$);

\node[above=2mm of $(left.east)!0.5!(right.west)$]
    {$\bar{\phi}(x)$};

\end{tikzpicture}

%% file: figures/diagram.tex
\begin{tikzpicture}[
    >=stealth,
    node distance=1.5cm,
    block/.style={
        draw,
        thick,
        minimum height=1.2cm,
        align=center,
        font=\large
    },
    line/.style={thick,->}
]

\node[block, minimum width=2.0cm] (lift) {$\bar{\phi}_{\theta^*}(\cdot)$};

\node[block, minimum width=4.0cm, right=2.0cm of lift] (mpc)
{$\begin{array}{c}
\text{Robust MPC}\\
\text{Problem~\ref{Prob:robust_MPC}}
\end{array}$};

\node[block, minimum width=4.0cm, right=1.8cm of mpc] (geo)
{$\begin{array}{c}
\text{Geodesic}\\
\text{Equation \eqref{eqn:geodesic_KOOPMAN}}
\end{array}$};

\node[block, minimum width=3.5cm, below=2.0cm of geo] (feedback)
{$
u_k=\bar{u}^*_{0|k}+
\displaystyle\int_0^1
K( \gamma_{0|k}^*(s))
\frac{\partial\gamma_{0|k}^*(s)}{\partial s}\,ds
$};

\node[block, minimum width=3.2cm, left=2.4cm of feedback] (plant)
{$x_{k+1}=f(x_k,u_k)$};

\draw[line] (lift) -- node[above] {$z_k$} (mpc);
\draw[line] (mpc) -- node[above] {$\bar{z}^*_{0|k}$} (geo);

\draw[line] (mpc.south) -- ++(0,-1.2)
node[pos=0.45,right] {$\bar{u}^*_{0|k}$}
-| (feedback.north);

\draw[line] (geo.east) -- ++(1.7,0)
-- ++(0,-3.27)
-- (feedback.east)
node[pos=0.25,right,xshift=5pt,yshift=55pt] {$ \gamma_{0|k}^*(s)$};

\draw[line] ([xshift=5pt]lift.east) -- ++(0,1.3)
-- (geo.north |- {$(lift.east)+(0,1.3)$})
node[pos=0.35,above] {$z_k$}
-- (geo.north);

\draw[line] (feedback.west) -- node[above] {$u_k$} (plant.east);

\draw[line] (plant.west) -- ++(-1.8,0)
-- (lift.south)
node[pos=0.5,above,xshift=10pt,yshift=3pt] {$x_k$};

\end{tikzpicture}